\documentclass[smallcondensed]{svjour3}     
\smartqed  

\usepackage{graphicx}
\usepackage[english]{babel}
\usepackage{graphics}
\usepackage{epsfig}
\usepackage{amssymb}
\usepackage{amsmath}
\usepackage{amsfonts}
\usepackage{float}

\newcommand{\B}[1]{\mbox{\boldmath $#1$}}
\newcommand{\be}{\begin{equation}}
\newcommand{\ee}{\end{equation}}
\newcommand{\ba}{\begin{array}}
\newcommand{\ea}{\end{array}}

\newcommand{\al}{{\alpha}}
\newcommand{\bt}{{\beta}}
\newcommand{\G}{{\Gamma}}
\newcommand{\g}{{\gamma}}

\newcommand{\Om}{\Omega}

\newcommand{\sg}{\sigma}
\newcommand{\tl}{\tilde}

 \journalname{Calcolo}
 
\begin{document}

\title{A Real QZ Algorithm for Structured Companion Pencils
\thanks{This work was partially supported by GNCS--INDAM and University of Pisa.}
}

\author{P. Boito         \and
        Y. Eidelman  \and
        L. Gemignani
}

\institute{P. Boito \at
              XLIM--DMI UMR CNRS 7252 Facult\'e des Sciences et Techniques, 
 123 avenue A. Thomas, 87060 Limoges, France \\
 and CNRS, Universit\'e de Lyon, Laboratoire LIP (CNRS, ENS Lyon, Inria, UCBL), 46 all\'ee d'Italie, 69364 Lyon Cedex 07, France\\
              \email{paola.boito@unilim.fr}           
           \and
           Y. Eidelman \at
              School of Mathematical Sciences, Raymond and Beverly
Sackler Faculty of Exact Sciences, Tel-Aviv University, Ramat-Aviv,
69978, Israel \\
\email{eideyu@post.tau.ac.il}
       \and
       L. Gemignani \at
       Dipartimento di Informatica, Universit\`a di Pisa,
Largo Bruno Pontecorvo 3, 56127 Pisa, Italy \\
\email{l.gemignani@di.unipi.it}
}

\date{Received: date / Accepted: date}

\maketitle

\begin{abstract}
We design a fast implicit  real QZ algorithm for eigenvalue computation of  structured companion pencils 
 arising from   linearizations of  polynomial rootfinding problems.  The modified  QZ  algorithm computes the
generalized   eigenvalues of  an  $N\times N$   
 structured matrix pencil
 using $O(N)$ flops per iteration and $O(N)$ memory storage.
Numerical experiments  and comparisons confirm the effectiveness and the
stability of the proposed  method.
\end{abstract}

\section{Introduction}\label{sec:intro}
\setcounter{equation}{0}
Linearization techniques based on polynomial interpolation  are becoming nowadays  a standard way to solve numerically 
nonlinear zerofinding problems for polynomials or more 
generally for analytic functions \cite{AKT}. Since in many applications the interest is in the approximation of real zeros, 
 methods using Chebyshev--like expansions are usually employed. Alternatively,  Lagrange interpolation at the roots of unity 
can be considered.  For a real function a straightforward modification of the classical approach \cite{Cor,law}  yields a 
structured companion pencil $\mathcal A(\lambda)=F-\lambda G$ where $F$ and $G$ are  real $N\times N$ low rank corrections of unitary matrices.
The computation of the generalized eigenvalues of this pencil can be performed by means of the  QZ algorithm \cite{WD} 
suitably adjusted to work in 
real arithmetic.  

In this paper we propose  a fast adaptation of the real QZ algorithm for computing the generalized eigenvalues of certain 
$N\times N$ structured  pencils using only $O(N)$ flops per iteration and $O(N)$ memory storage. Since in most cases it is reasonable to assume that the total number of iterations is a small  multiple of $N$ (see e.g., \cite{W_eig}), we have a heuristic complexity estimate of $O(N^2)$ flops to compute all the eigenvalues.

The pencils  $\mathcal A(\lambda)=F-\lambda G$, 
$F, G\in \mathbb R^{N\times N}$,
 we consider here satisfy two basic properties:
\begin{enumerate}
\item $F$ is upper Hessenberg and $G$ is  upper triangular;
\item $F$ and $G$ are rank--one corrections of unitary matrices.
\end{enumerate} 
We refer to  such  a   pencil  $\mathcal A(\lambda)$ as a companion--like  pencil,   since   the class includes companion pencils as a special case.
Sometimes $\mathcal A(\lambda)$ is also denoted by $(F,G)\in \mathbb R^{N\times N}\times \mathbb R^{N\times N}$.

Let $(F_k,G_k)$, $k\geq 0$, $F_0=F, G_0=G$, be the sequence of matrix pairs (pencils) 
generated by the real QZ algorithm starting from the companion--like pencil 
$\mathcal A(\lambda)$. Single or double shifting is applied in the generic iteration $F_k\rightarrow F_{k+1}$ $G_k\rightarrow G_{k+1}$  in order to 
carry out all the computations in real arithmetic.  Whatever strategy is used, it is found that both $\mathcal A_k(\lambda)$ and 
$\mathcal A_{k+1}(\lambda)$ are still companion--like pencils. As a consequence of this  invariance we obtain that all the matrices involved in the 
QZ iteration inherit a rank structure in their upper triangular parts. This makes it possible to represent $F_k, G_k$ and $F_{k+1}, G_{k+1}$ as data--sparse 
matrices specified  by a number of parameters (called generators) which is  linear  w.r.t. the size of the matrices. This general principle has been applied, for instance, in \cite{AMRVW} and \cite{last}.

In this paper we  introduce a convenient set of generators 
and design a  structured  variant of the real QZ 
iteration  which takes in input the generators of $F_k$ and $G_k$ together with the shift parameters and returns  as output the generators of 
$F_{k+1}$ and $G_{k+1}$.  It is shown that the arithmetic cost for each iteration is $O(N)$ using linear memory storage.  Numerical experiments 
confirm the effectiveness and the robustness of the resulting  eigensolver.

The paper is organized as follows. In Section \ref{sec:2}
we set up the scene by introducing the matrix  problem and its basic properties.
In Section \ref{sec:3} we  define an appropriate set of generators for the matrices involved. 
In Section \ref{sec:4} we design the fast adaptation of the QZ algorithm using these generators and exploiting 
the resulting data--sparse representations.  We focus here on double shifting, since the single shifted iteration 
has been already devised in \cite{last}.  A proof  of the correctness of the algorithm is given in Appendix. 
Finally, in Section \ref{sec:5} we show the results of numerical experiments,  whereas conclusion and future work are presented in Section  \ref{sec:6}.

\section{The Problem Statement}\label{sec:2}
\setcounter{equation}{0}
Companion pencils and  companion--like  pencils expressed in the 
Lagrange basis  at the roots of unity are specific instances of the 
following general class. 

\begin{definition}\label{pn}
The matrix pair $(A,B)$,  $A, B\in \mathbb R^{N\times N}$, belongs to the
class 
$\mathcal P_N \subset \mathbb R^{N\times N} \times \mathbb R^{N\times N}$   
of companion--like  pencils 
iff:
\begin{enumerate}
\item $A\in \mathbb R^{N\times N}$ is upper Hessenberg;
\item $B\in \mathbb R^{N\times N}$ is upper triangular;
\item There exist two vectors $\B z\in \mathbb R^N$ and $\B w\in \mathbb R^N$ 
and an orthogonal matrix $V\in \mathbb R^{N\times N}$  such that
\begin{equation}\label{a}
 A=V-\B z \B w^*; 
\end{equation}
\item There exist two vectors $\B p\in\mathbb R^N$ and $\B q\in\mathbb R^N$ and
an orthogonal matrix $U\in \mathbb R^{N\times N}$  such that
\begin{equation}\label{b}
 B=U-\B p \B q^*.
\end{equation}
\end{enumerate}
\end{definition}

In order to characterize the individual properties of the matrices $A$ and $B$ 
we give some additional definitions.

\begin{definition}\label{bn}
We denote  by  ${\mathcal T}_{N}$ the class of  upper triangular matrices 
$B\in \mathbb R^{N\times N}$ that are rank-one perturbations of orthogonal
matrices, i.e., such that \eqref{b} holds for a suitable orthogonal matrix $U$ 
and vectors $\B p, \B q$.
\end{definition}

Since $B$ is upper triangular the strictly lower triangular part of the 
orthogonal matrix $U$ in \eqref{b} coincides with the corresponding part of the
rank one matrix $\B p \B q^*$, i.e.,
\be\label{semop23}
U(i,j)= p(i) q^*(j),\quad 1\le j<i\le N,
\end{equation}
where $\{p(i)\}_{i=1,\ldots,N}$ and $\{q(j)\}_{j=1,\ldots,N}$ are the entries 
of $\B p$ and $\B q$, respectively.

\begin{definition}\label{un}
We denote by ${\mathcal U}_{N}$ the class of orthogonal matrices 
$U\in \mathbb R^{N\times N}$ that satisfy the condition \eqref{semop23}, i.e.,
for which there exist vectors $\B p, \B q$ such that
the matrix $B=U-\B p \B q^*$ is an upper triangular matrix.
\end{definition}

Observe that we have
\[
U\in {\mathcal U}_{N} \Rightarrow
{\rm rank}\, U(k+1\colon N,1\colon k)\le 1, \quad k=1,\dots,N-1.
\]
From the nullity theorem \cite{FM}, see also \cite[p.142]{EGH1}, it follows 
that the same  property also holds in the strictly upper triangular part, 
namely,
\begin{equation}\label{setp26}
U\in {\mathcal U}_{N} \Rightarrow
{\rm rank}\, U(1\colon k, k+1 \colon N)\le 1, \quad k=1,\dots,N-1.
\end{equation}

\begin{definition}\label{hn}
We denote  by  ${\mathcal H}_{N}$ the class of  upper Hessenberg  matrices 
$A\in \mathbb R^{N\times N}$ that are rank one perturbations of orthogonal
matrices, i.e., such that \eqref{a} holds for a suitable orthogonal matrix $V$ 
and vectors $\B z, \B w$.
\end{definition}

\begin{definition}\label{qn}
We denote by ${\mathcal V}_{N}$ the class of orthogonal matrices 
$V\in \mathbb R^{N\times N}$ for which there exist vectors $\B z,\B w$ such 
that the matrix $A=V-\B z \B w^*$ is an upper Hessenberg matrix.
\end{definition}

We find that 
\[
V\in {\mathcal V}_{N} \Rightarrow
{\rm rank}\, V(k+2\colon N,1\colon k)\le 1, \quad k=1,\dots,N-2.
\]
Again from the nullity theorem  it follows that a similar  property also
holds in the upper triangular part, namely,
\begin{equation}\label{setp27}
V\in {\mathcal V}_{N} \Rightarrow
{\rm rank}\, V(1\colon k, k \colon N)\le 2, \quad k=1,\dots,N.
\end{equation}

In this paper we consider the problem of efficiently computing the (generalized) eigenvalues of a 
companion--like matrix pencil $ (A,B)\in \mathcal P_N$ by exploiting the rank  and banded structures 
of the matrix classes mentioned above.
The QZ algorithm  is  the customary method for solving generalized eigenvalue 
problems numerically  by means of unitary  transformations
(see e.g. \cite{GVL} and \cite{WD}). 
For the pair $(A,B)$  in Hessenberg/triangular form the implicit QZ step consists in the computation of 
unitary matrices $Q$ and $Z$ such that 
\be\label{sten21}
A_1=Q^*AZ\;\mbox{is upper Hessenberg},\;
B_1=Q^*BZ\;\mbox{is upper triangular}
\end{equation} 
and some initial conditions hold.
For the QZ  iteration  applied to a real matrix pair with  double shifting  the initial condition is
\be\label{semer22d}
(Q^*p(A B^{-1}))(:,2:N)=0, 
\end{equation}
where $p(z)=\alpha + \beta z +\gamma z^2 $ is the shift polynomial.
In this case one obtains the orthogonal Hessenberg matrices $Q$ and $Z$ in the 
form 
\be\label{smap22d}
Q=\tl Q_1\tl Q_2\cdots\tl Q_{N-2}\tl Q_{N-1},\quad 
Z=\tl Z_1\tl Z_2\cdots\tl Z_{N-2}\tl Z_{N-1},
\end{equation}  
where 
\be\label{smip22d}
\begin{gathered}
\tl Q_i=I_{i-1}\oplus Q_i\oplus I_{N-i-2},\;i=1,\dots,N-2,\quad 
\tl Q_{N-1}=I_{N-2}\oplus Q_{N-1},\\
\tl Z_i=I_{i-1}\oplus Z_i\oplus I_{N-i-2},\;i=1,\dots,N-2,\quad 
\tl Z_{N-1}=I_{N-2}\oplus Z_{N-1}
\end{gathered}
\end{equation}
and $Q_i,Z_i,\;i=1,\dots,N-2$ are $3\times3$ orthogonal matrices, $Q_{N-1},
Z_{N-1}$ are real Givens rotations.

Since the Hessenberg/triangular 
form is preserved under the QZ iteration  an easy computation then yields
\begin{equation}\label{qzss}
(A,B)\in \mathcal P_N, \ (A,B)\overset{{\rm QZ \ step}} 
\rightarrow (A_1,B_1) \Rightarrow (A_1,B_1)\in \mathcal P_N.
\end{equation}
Indeed if $Q$ and $Z$ are unitary then from (\ref{a}) and (\ref{b}) it follows
that the matrices $A_1=Q^*AZ$ and $B_1=Q^*BZ$ satisfy the relations
$$
A_1=V_1-\B z_1\B w_1^*,\quad B_1=U_1-\B p_1\B q_1^*
$$
with the unitary matrices $V_1=Q^*VZ,\;U_1=Q^*UZ$ and the vectors
$\B z_1=Q^*z,\;\B w_1=Z^*w,\;\B p_1=Q^*p,\;\B q_1=Z^*q$. Moreover one can 
choose the unitary matrices $Q$ and $Z$ such that the matrix $A_1$ is upper
Hessenberg and the matrix $B_1$ is upper triangular.
Thus, one  can in principle think of  designing a structured QZ iteration 
that, given in input a condensed representation of the matrix pencil  
$(A,B)\in \mathcal P_N$, returns as output a condensed representation of  
$(A_1,B_1)\in \mathcal P_N$ generated by one step of the classical QZ algorithm
applied to $(A,B)$. In the next sections we first  introduce an  eligible  
representation  of a  rank-structured matrix pencil $(A,B)\in \mathcal P_N$ and
then  discuss the modification of this representation under the QZ process.

\section{Quasiseparable Representations}\label{sec:3}
\setcounter{equation}{0}
In this  section we exploit the properties of  quasiseparable 
representations of rank--structured matrices \cite{EGfirst},
\cite[Chapters 4,5]{EGH1}.
First  we recall some general  results and  definitions. Subsequently, 
we describe their  adaptations 
for the representation of the matrices involved in the structured QZ 
iteration applied to an input 
matrix pencil $(A,B)\in \mathcal P_N$.

A matrix $M=\{M_{ij}\}_{i,j=1}^N$ is {\em $(r^L,r^U)$-quasiseparable}, with 
$r^L, r^U$ positive integers, if, using MATLAB\footnote{MATLAB is a registered 
trademark of The Mathworks, Inc..} notation, 
\begin{gather*}
\max_{1\leq k\leq N-1}{\rm rank}\,(M(k+1:N,1:k))\leq r^L,\\
\max_{1\leq k\leq N-1}{\rm rank}\,(M(1:k,k+1:N))\leq r^U.
\end{gather*}
Roughly speaking, this means that every submatrix extracted from the lower 
triangular part of $M$ has rank at most $r^L$, and every submatrix extracted 
from the upper triangular part of $M$ has rank at most $r^U$. Under this 
hypothesis, $M$ can be represented using $\mathcal{O}(((r^L)^2+(r^U)^2) N)$ parameters. 
In this section we present such a representation.

The quasiseparable representation of a rank--structured matrix consists of a 
set of vectors and matrices used to generate its  entries. For the sake of 
notational simplicity and clarity, generating  matrices  and vectors are  denoted by a 
roman lower-case  letter.
 
In this representation, the entries of $M$ take the form
 
 \begin{equation}\label{qs1}
M_{ij}=\left\{\begin{array}{ll}
p(i)a_{ij}^{>}q(j),&1\le j<i\le N,\\
d(i),&1\le i=j\le N,\\
g(i)b_{ij}^{<}h(j),&1\le i<j\le N\end{array} \right.
\end{equation}
 where:
 \begin{itemize}
 \item[-] $p(2),\ldots,p(N)$ are row vectors of length $r^L$,
  $q(1),\ldots,q(N-1)$ are column vectors of length $r^L$, and
  $a(2),\ldots,a(N-1)$ are matrices of size $r^L\times r^L$; these are called 
{\em lower quasiseparable generators} of order $r^L$;
 \item[-] $d(1),\ldots,d(N)$ are numbers (the diagonal entries),
 \item[-] $g(2),\ldots,g(N)$ are row vectors of length $r^U$,
  $h(1),\ldots,h(N-1)$ are column vectors of length $r^U$, and
  $b(2),\ldots,b(N-1)$ are matrices of size $r^U\times r^U$; these are called 
{\em upper quasiseparable generators} of order $r^U$;
\item[-]
 the matrices $a_{ij}^{>}$ and $b_{ij}^{<}$ are defined as
 \[
\left\{\begin{array}{ll}a_{ij}^{>}=a(i-1)\cdots
a(j+1) \  {\rm for}\  i>j+1;\\ a_{j+1,j}^{>}=1
\end{array}\right.
\]
and
\[
\left\{\begin{array}{ll}
b_{ij}^{<}=b(i+1)\cdots
b(j-1) \ {\rm  for} \  j>i+1; \\
b_{i,i+1}^{<}=1.
\end{array}\right.
\]
 \end{itemize}

From (\ref{setp26}) it follows that any matrix from the class 
${\mathcal U}_{N}$ has upper quasiseparable generators with orders equal to one.

 The quasiseparable representation can be generalized to the case where $M$ is 
a block matrix, and to the case where the generators do not all have the same 
size, provided that their product is well defined. Each block $M_{ij}$ of size 
$m_i\times n_j$ is represented as in \eqref{qs1}, except that the sizes of the 
generators now depend on $m_i$ and $n_j$, and possibly on the index of $a$ and 
$b$. More precisely:
 \begin{itemize}
 \item[-]
 $p(i), q(j), a(k)$ are matrices of sizes $m_i\times
r^L_{i-1},\;r^L_j\times n_j,\; r^L_k\times r^L_{k-1}$,
respectively; 
\item[-]
$d(i)\;(i=1,\dots,N)$ are $m_i\times n_i$ matrices,
\item[-]
$g(i),h(j),b(k)$ are matrices of sizes
$m_i\times r^U_i,\;r^U_{j-1}\times n_j,\; r^U_{k-1}\times r^U_k$,
respectively.
 \end{itemize}
 The numbers $r^L_k,r^U_k\;(k=1,\dots,N-1)$ are called the {\it orders}
of these generators.
 
It is worth noting that lower and upper quasiseparable  generators of a matrix 
are not uniquely defined. A set of generators with minimal orders can be 
determined according to the ranks of maximal submatrices  located in  the lower
and upper triangular parts of  the matrix.

One advantage of the block representation for the purposes of the present paper
consists in the fact that $N\times N$ upper Hessenberg matrices can be treated 
as $(N+1)\times (N+1)$ block upper triangular ones by choosing block sizes as 
\be\label{aprmn18s}
m_1=\dots=m_N=1,\;m_{N+1}=0,\quad
n_1=0,\;n_2=\dots=n_{N+1}=1.
\end{equation}
Such a treatment allows also to consider quasiseparable representations which 
include the main diagonals of matrices. Assume that $C$ is an $N\times N$ 
scalar
matrix with the entries in the upper triangular part represented in the form
\be\label{appr18}
C(i,j)=g(i)b^<_{i-1,j}h(j),\quad 1\le i\le j\le N
\end{equation}
with matrices $g(i),h(i)\;(i=1,\dots,N),\;b(k)\;(k=1,\dots,N-1)$ of sizes
$1\times r_i,r_i\times1,r_k\times r_{k+1}$. The elements $g(i),h(i)\;
(i=1,\dots,N),\;b(k)\;(k=1,\dots,N-1)$ are called {\it upper triangular
generators} of the matrix $C$ with orders $r_k\;(k=1,\dots,N)$. 
From (\ref{setp27}) it follows that any matrix from the class 
${\mathcal V}_{N}$ has upper triangular generators
with orders not greater than two. 
If we treat a matrix $C$ as a block one with entries if sizes (\ref{aprmn18s}) 
we conclude
that the elements $g(i)\;(i=1,\dots,N),\;h(j-1)\;(j=2,\dots,N+1),\;b(k-1)\;
(k=2,\dots,N)$ are upper quasiseparable generators of $C$. 

Matrix operations involving zero-dimensional arrays (empty matrices) are 
defined according to the rules used in MATLAB and described in \cite{deBoor}.
In particular, the  product of
a  $m\times 0$ matrix by a $0\times m$ matrix is a $m\times m$ matrix with all 
entries equal to 0.
Empty matrices may be used in assignment statements as a convenient way to add 
and/or delete rows or columns of matrices.

\subsection{{\bf Representations of matrix pairs from the class
${\mathcal P}_{N}$}}\label{psec}

Let $(A,B)$ be a matrix pair from the class ${\mathcal P}_{N}$.
The corresponding matrix $A$ from the class ${\mathcal H}_{N}$ is completely
defined by the following parameters:
\begin{enumerate}
\item  the subdiagonal entries $\sg^A_k\;(k=1,\dots,N-1)$ of the matrix $A$;  
\item the upper triangular generators $g_V(i),h_V(i)\;(i=1,\dots,N),\;
b_V(k)\;(k=1,\dots,N-1)$ of the corresponding unitary matrix $V$ from the class
${\mathcal V}_N$;
\item  the vectors of perturbation $\B z={\rm col}(z(i))_{i=1}^N,\;
\B w={\rm col}(w(i))_{i=1}^N$.
\end{enumerate}
From  \eqref{setp27} it follows that 
the matrix $V\in \mathcal V_N$ has upper triangular generators with orders
not greater than two.

The corresponding matrix $B$ from the class ${\mathcal T}_{N}$ is completely
defined by the following parameters:
\begin{enumerate}
\item  the diagonal entries $d_B(k)\;(k=1,\dots,N)$ of the matrix $B$;  
\item the upper quasiseparable generators $g_U(i)\;(i=1,\dots,N-1),\;
h_U(j)\;(j=2,\dots,N)$,
$b_U(k)\;(k=2,\dots,N-1)$ of the corresponding unitary 
matrix $U$ from the class ${\mathcal U}_N$;
\item  the vectors of perturbation $\B p={\rm col}(p(i))_{i=1}^N,\;
\B q={\rm col}(q(i))_{i=1}^N$.
\end{enumerate}
From  \eqref{setp26} it follows that 
the matrix $U\in \mathcal U_N$ has upper quasiseparable generators with orders
equal one.

All the given parameters define completely the matrix pair $(A,B)$ from the 
class $\mathcal P_N$. Updating of these parameters while keeping the minimal
orders of generators is a task of the fast QZ  iteration described in the next section.

\section{A fast implicit double shifted  QZ iteration via generators}\label{sec:4}
\setcounter{equation}{0}

In this section we present our fast adaptation of the double--shifted QZ algorithm for a 
matrix pair $(A,B)\in \mathcal P_N$.  The algorithm takes in input a quasiseparable representation of the 
matrices $A$ and $B$ together with the coefficients of the real quadratic shift polynomial and it returns as output a 
possibly not minimal quasiseparable representation of the matrices $(A_1,B_1)\in \mathcal P_N$ such that 
\eqref{qzss} holds. The algorithm computes the unitary matrices $Q_i$ and $Z_i$ defined in 
\eqref{smip22d}. It  basically splits into  the following four stages:
\begin{enumerate}
\item a {\bf preparative phase} where $Q_1$  is found so as to satisfy the shifting condition;
\item the {\bf chasing the bulge} step where the unitary matrices $Q_2, \ldots, Q_{N-2}$ and 
$Z_1, \ldots, Z_{N-3}$ are determined in such a way to  perform the Hessenberg/triangular reduction procedure;
\item a {\bf closing} phase where the last  three transformations $Q_{N-1}$, $Z_{N-2}$ and $Z_{N-1}$ are carried out; 
\item the final stage of {\bf recovering the generators} of the updated  pair.
\end{enumerate}
For the sake of brevity the stage 2 and 3 are grouped together  by using empty and zero quantities when needed.
 The correctness of the algorithm is proved in the Appendix.  
Some technical details concerning shifting  strategies and shifting 
techniques are  discussed in the section on numerical experiments. Compression of generators 
yielding minimal representations can be achieved by using the  methods  devised in 
\cite{last}. The incorporation of these compression  schemes does not alter the complexity of the main algorithm 
shown below. 
 
 \vspace{10pt}
 
\centerline{{\bf ALGORITHM}: Implicit QZ iteration for companion--like pencils with double shift}

\begin{enumerate}

\item {\bf INPUT}:
\begin{enumerate}
\item the subdiagonal entries
$\sg^A_k\;(k=1,\dots$,$N-1)$ of the matrix $A$;
\item  the upper triangular generators 
$g_V(i),h_V(i)\;(i=1,\dots,N),\;b_V(k)\;(k=1,\dots,N-1)$ with orders 
$r_k^V\;(k=1,\dots,N)$ of the matrix $V$;
\item  the diagonal entries $d_B(k)\;
(k=1,\dots,N)$ of the matrix $B$;
\item the  upper quasiseparable generators 
$g_U(i)\;(i=1,\dots,N-1),\;h_U(j)\;(j=2,\dots,N),\;b_U(k)\;
(k=2,\dots,N-1)$ with orders $r^U_k\;(k=1,\dots,N-1)$ of the matrix $U$;
\item  the perturbation vectors 
$\B z={\rm col}(z(i))_{i=1}^N$, $\B w={\rm col}(w(i))_{i=1}^N$, 
$\B p={\rm col}(p(i))_{i=1}^N$, $\B q={\rm col}(q(i))_{i=1}^N$;
\item the coefficients of the shift polynomial  $p(z)=\alpha + \beta z +\gamma z^2 \in \mathbb R[z]$;
\end{enumerate}
\item {\bf OUTPUT}:
\begin{enumerate}
\item the subdiagonal entries
$\sg^{A_1}_k\;(k=1,\dots,N-1)$ of the matrix $A_1$; 
\item  upper triangular generators
$g^{(1)}_V(i),h^{(1)}_V(i)\;(i=1,\dots,N),\;b^{(1)}_V(k)\;(k=1,\dots,N-1)$ 
 of the matrix $V_1$;
\item the  diagonal entries $d^{(1)}_B(k)\;(k=1,\dots,N)$ of the matrix $B_1$;
\item  upper quasiseparable 
generators $g^{(1)}_U(i)\;(i=1,\dots,N-1),\;h^{(1)}_U(j)\;(j=2,\dots,N),\;
b^{(1)}_U(k)\;(k=2,\dots,N-1)$ of the matrix $U_1$;
\item perturbation vectors 
$\B z_1={\rm col}(z^{(1)}(i))_{i=1}^N,\;\B w_1={\rm col}(w^{(1)}(i))_{i=1}^N,\;
\B p_1={\rm col}(p^{(1)}(i))_{i=1}^N,\;\B q_1={\rm col}(q^{(1)}(i))_{i=1}^N$;
\end{enumerate}
\item{\bf COMPUTATION}:
\begin{itemize}
\item {\bf Preparative Phase}
\begin{enumerate}
\item  Compute $\B s=(p(A B^{-1})\B e_1)(1\colon 3)$ 
and determine the $3\times3$ orthogonal matrix $Q_1$ from the condition
\be\label{fe218}
Q_1^*\B s=
\left(\ba{ccc}\times&0&0\ea\right)^*.
\end{equation}
\item Compute
\be\label{parq18d}
\left(\ba{c}\tl g_V(3)\\\beta^V_3\ea\right)=
Q_1^*
\left(\ba{ccc}g_V(1)h_V(1)&g_V(1)b_V(1)h_V(2)&g_V(1)b_V(1)b_V(2)\\
\sg^V_1&g_V(2)h_V(2)&g_V(2)b_V(2)\\z(3)w(1)&\sg_2^V&g_V(3)\ea\right)
\end{equation}

and determine the matrices $f^V_3,\phi^V_3$ of sizes $2\times2,2\times r^V_3$ from 
the partition
\be\label{pra182d}
\beta^V_3=\left[\ba{cc}f^V_3&\phi^V_3\ea\right].
\end{equation}
\item Compute 
\be\label{feura21}
\left(\ba{c}z^{(1)}(1)\\\chi_3\ea\right)=
Q_1^* \left(\ba{c} z(1)\\z(2)\\z(3)\ea\right), 
\quad 
\gamma_2=\left(\ba{c} w(1)\\w(2)\ea\right)
\end{equation}
with the number $z^{(1)}$ and two-dimensional columns $\chi_3,\gamma_2$.
 Compute
\be\label{janj11f}
f^A_3=f^V_3-\chi_3\gamma_2^*,\quad \varphi_3^A=\phi^V_3.
\end{equation}
\item Set
\be\label{oct141d}
c_2=\left(\ba{c}p(1)\\p(2)\ea\right),\quad 
\theta_1=q(1),\;\theta_2=\left(\ba{c}q(1)\\q(2)\ea\right).
\end{equation}
Compute
\be\label{lyutf22d}
d_U(1)=d_B(1)+p(1)q(1),\quad d_U(2)=d_B(2)+p(2)q(2)
\end{equation} 
and set
\be\label{octo14.1d}
f^U_2=\left(\ba{cc}d_U(1)&g_U(1)h_U(2)\\p(2)q(1)&d_U(2)\ea\right),\;
\phi^U_2=\left(\ba{c}g_U(1)b_U(2)\\g_U(2)\ea\right),
\end{equation}
\be\label{liufm12}
\varepsilon=g_U(1)h_U(2)-p(1)q(2),
\end{equation}
\be\label{oct14.1d}
f^B_2=\left(\ba{cc}d_B(1)&\varepsilon\\0&d_B(2)\ea\right),\;
\varphi^B_2=\phi^U_2.
\end{equation}
\end{enumerate}
\item {\bf Chasing the Bulge} For $k=1,\dots,N-1$ perform the following:
\begin{enumerate}
\item {\em (Apply $Q_k$ and determine $Z_k$).} Compute the two-dimensional column $\varepsilon_{k+1}^B$ via
\be\label{jll8od}
\varepsilon^B_{k+1}=\varphi^B_{k+1}h_U(k+2)-c_{k+1}q(k+2),
\end{equation}
and the $3\times3$ matrix $\Phi_k$ by the formula 
\be\label{lmay2ujod}
\Phi_k=Q_k^*\left(\ba{cc}f^B_{k+1}&\varepsilon^B_{k+1}\\0&d_B(k+2)\ea\right).
\end{equation}
 Determine the $3\times3$ orthogonal matrix $Z_k$ such that
\be\label{imay2ujod}
\Phi_k(2:3,:)Z_k=\left(\ba{ccc}0&\times&\times\\0&0&\times\ea\right).
\end{equation}
\item  {\em (Determine $Q_{k+1}$).} Compute the column
\be\label{ep}
\epsilon^A_{k+2}=\varphi_{k+2}^Ah_V(k+2)-\chi_{k+2}w(k+2)
\end{equation}
and the $3\times3$ matrix $\Om_k$ by the formula
\be\label{aprl18d}
\Om_k=\left(\ba{cc}f^A_{k+2}&\epsilon^A_{k+2}\\0&\sg^A_{k+2}\ea\right)Z_k.
\end{equation}
Determine the $3\times3$ orthogonal matrix $Q_{k+1}$ and the number
$(\sg^A_k)^{(1)}$ such that
\be\label{aprle18d}
Q^*_{k+1}\Om_k(:,1)=\left(\ba{c}(\sg^A_k)^{(1)}\\0\\0\ea\right).
\end{equation}
\item  {\em (Update generators for $U$ and $B$).} Compute
\be\label{na}
d_U(k+2)=d_B(k+2)+p(k+2)q(k+2),
\end{equation}
\be\label{liufre20d}
\begin{gathered}
\left(\ba{cc}\tl d_U(k+2)&\tl g_U(k+2)\\\times&\beta^U_{k+2}\ea\right)
=Q_k^* \ \tilde U_k \ \left(\ba{cc}Z_k&0\\0&I_{r^U_{k+2}}\ea\right)
\end{gathered}
\end{equation}
where 
\[
\tilde U_k=\left(\ba{ccc}f^U_{k+1}&\phi^U_{k+1}h_U(k+2)&\phi^U_{k+1}b_U(k+2)\\
p(k+2)\theta^*_{k+1}&d_U(k+2)&g_U(k+2)\ea\right)
\]
and determine the matrices $f^U_{k+2},\phi^U_{k+2}$ of
sizes $2\times2,2\times r^U_{k+2}$ from the partition
\be\label{fevra20d}
\bt^U_{k+2}=\left[\ba{cc}f^U_{k+2}&\phi^U_{k+2}\ea\right].
\end{equation}
Compute
\be\label{liufg20d}
\begin{gathered}
\left(\ba{cc}\tl h_U(k+2)&\tl b_U(k+2)\ea\right)=\\
\left(\ba{ccc}I_2&0&0\\0&h_U(k+2)&b_U(k+2)\ea\right)
\left(\ba{cc}Z_k&0\\0&I_{r^U_{k+2}}\ea\right).
\end{gathered}
\end{equation}
Compute
\be\label{feura21d}
\begin{gathered}
\left(\ba{c}p^{(1)}(k)\\c_{k+2}\ea\right)=
Q^*_k\left(\ba{c}c_{k+1}\\p(k+2)\ea\right) \\
\left(\ba{c}q^{(1)}(k)\\\theta_{k+2}\ea\right)=
Z^*_k\left(\ba{c}\theta_{k+1}\\q(k+2)\ea\right)
\end{gathered}
\end{equation}
with the numbers $p^{(1)}(k),q^{(1)}(k)$ and two-dimensional columns 
$c_{k+2},\theta_{k+2}$.
 Compute
\be\label{fegre21d}
f^B_{k+2}=f^U_{k+2}-c_{k+2}\theta^*_{k+2},\quad\varphi^B_{k+2}=\phi^U_{k+2}.
\end{equation}
\item {\em (Update generators for $V$ and $A$).} Compute
\be\label{jll8o}
\sg^V_{k+2}=\sg^A_{k+2}+z(k+3)w(k+2),
\end{equation}
\be\label{leap18d}
\begin{gathered}
\left(\ba{cc}\tl d_V(k+3)&\tl g_V(k+3)\\\times&\beta^V_{k+3}\ea\right)=\\
Q^*_{k+1} \ \tilde V_{k+2} \
\left(\ba{cc}Z_k&0\\0&I_{r^V_{k+3}}\ea\right),
\end{gathered}
\end{equation}
where 
\[
\tilde V_{k+2}=\left(\ba{ccc}f^V_{k+2}&\phi^V_{k+2}h_V(k+2)&\phi^V_{k+2}b_V(k+2)\\
z(k+3)\gamma^*_{k+1}&\sg^V_{k+2}&g_V(k+3)\ea\right).
\]
Determine the matrices $f^V_{k+3},\phi^V_{k+3}$ of sizes 
$2\times2,2\times r^V_{k+3}$  from the partition
\be\label{pra18d}
\beta^V_{k+3}=\left[\ba{cc}f^V_{k+3}&\phi^V_{k+3}\ea\right].
\end{equation}
Compute
\be\label{mmay3hbd}
\begin{gathered}
\left(\ba{cc}\tl h_V(k+3)&\tl b_V(k+3)\ea\right)=\\
\left(\ba{ccc}I_2&0&0\\0&h_V(k+2)&b_V(k+2)\ea\right)
\left(\ba{cc}Z_k&0\\0&I_{r^V_{k+3}}\ea\right),
\end{gathered}
\end{equation}
and
\be\label{oct14qqd}
\begin{gathered}
\left(\ba{c}z^{(1)}(k+1)\\\chi_{k+3}\ea\right)=
Q_{k+1}^*\left(\ba{c}\chi_{k+2}\\z(k+3)\ea\right) \\
\left(\ba{c}w^{(1)}(k)\\\gamma_{k+2}\ea\right)=
Z_k^*\left(\ba{c}\gamma_{k+1}\\w(k+2)\ea\right)
\end{gathered}
\end{equation}
with the numbers $z^{(1)}(k+1),w^{(1)}(k)$ and two-dimensional columns 
$\chi_{k+3},\gamma_{k+2}$.
 Compute
\be\label{janj1f}
f^A_{k+3}=f^V_{k+3}-\chi_{k+3}\gamma_{k+2}^*,\quad 
\varphi_{k+3}^A=\phi^V_{k+3}.
\end{equation}
\end{enumerate}

\item {\bf Recovering of generators}
\begin{enumerate}
\item Set
\begin{gather*}
g_V^{(1)}(i-2)=\tl g_V(i),\;i=3,\dots,N+2,\\
h_V^{(1)}(j-3)=\tl h_V(j),\;j=4,\dots,N+3,\\
b_V^{(1)}(k-3)=\tl b_V(k),\;j=4,\dots,N+2.
\end{gather*}
\item Set
\begin{gather*}
g_U^{(1)}(i-2)=\tl g_U(i),\;i=3,\dots,N+1,\\
h_U^{(1)}(j-2)=\tl h_U(j),\;j=4,\dots,N+2,\\
b_U^{(1)}(k-2)=\tl b_U(k),\;k=3,\dots,N+1,\\
d_U^{(1)}(k-2)=\tl d_U(k),\;k=3,\dots,N+2.
\end{gather*}
\end{enumerate}
\end{itemize}
\end{enumerate}
{\bf END}

\begin{remark}\label{compre}
The complete algorithm incorporates the compression technique 
introduced in \cite{last} to  further process the generators returned by the algorithm by computing final   upper 
quasiseparable  generators
$g^{(1)}_U(i)\;(i=1,\dots,N-1)),\;h^{(1)}_U(j)\;(j=2,\dots,N),\;b^{(1)}_U(k)\;
(k=2,\dots,N-1)$ with orders not greater than  one of the matrix $U_1$ and, moreover, 
upper triangular generators
$g^{(1)}_V(i),h^{(1)}_V(i)$
$(i=1,\dots,N),\;b^{(1)}_V(k)$ $(k=1,\dots,N-1)$ with
orders not greater than two of the matrix $V_1$.
\end{remark}

\begin{remark}\label{complexity}
It can be interesting to compare the complexity and timings of the above algorithm versus the single-shift version presented in \cite{last}. Roughly speaking, each iteration of double-shift Fast QZ requires about  twice as many floating-point operations as single-shift Fast QZ; however, the double-shift version works in real arithmetic, whereas the single-shift algorithm requires complex operations. So we can expect a double-shift iteration to be $\mu/2$ times faster than a single-shift one, where $\mu$ is the speedup factor of real vs. complex arithmetic. A na\"ive operation count suggests that a complex addition requires two real flops and a complex multiplication requires six real flops
. This yields on average $\mu\approx 4$, although in practice $\mu$ is more difficult to quantify; here, for practical purposes, we have used the experimental estimate given below.

For the computation of all eigenvalues, the double-shift algorithm is about $\rho\mu/2$ times faster than the single-shift version, where $\rho$ is the ratio between the number of iterations needed to approximate a single eigenvalue with double shift and the number of iterations per eigenvalue with single shift. In practice, $\rho$ is often close to $2$, because each double-shift iteration approximates two eigenvalues instead of a single one, so the total number of iterations will be cut by one half.

Experiments done on the same machine and configuration used for the Fortran tests in Section \ref{sec:5} gave the following results:
\begin{itemize}
\item
After testing on a large number of scalar $ax+y$ operations, the parameter $\mu$ was estimated at about $2$. We used scalar operations for consistency with the structure of the algorithm. It should be pointed out, however, that experimental estimates of $\mu$ may depend on the machine and on the way the operations are computed, because the weight of increased storage and bandwidth may become prominent. (The same experiment run on matrix-vector operations gives $\mu\approx 4$ as predicted by the operation count).
\item
For random polynomials we found $\rho\approx 2$, whereas in the case of cyclotomic polynomials the double-shift algorithm converged faster and $\rho$ was closer to $3$.
\item
Comparison on total running time showed double-shift QZ to be about twice as fast as the single-shift version in the case of random polynomials, and about three times as fast for cyclotomic polynomials, which is consistent with the discussion above.  
\end{itemize}   
\end{remark}

In the next section we report the results of numerical experiments to illustrate the performance of the algorithm.

\section{Numerical Results}\label{sec:5}
The  fast QZ algorithm for eigenvalue computation  of structured pencils described in the previous section
has been implemented in MATLAB and in Fortran 90.\footnote{Both implementations are available for download at\\ {\tt http://www.unilim.fr/pages\_perso/paola.boito/software.html}.} The program  deals with  real  companion--like pencils  by applying the QZ method 
with single or double shift and it returns as output  the list of real or complex conjugate paired  approximations of the  eigenvalues.

The design of a practical algorithm needs to account for various possible shifting strategies and deflation techniques. 
Deflation is an important concept in the practical implementation of the QR/QZ  iteration. Deflation amounts
to setting a small subdiagonal element of the Hessenberg matrix  $A$  to zero.
This is called deflation because it splits the Hessenberg/triangular  matrix  pair into two smaller
subproblems which may be independently refined further. We say that $a_{k+1,k}$ is negligible  if 
\[
|a_{k+1,k}|\leq {\tt u}(|a_{k+1,k+1}|+|a_{k,k}|), 
\]
and then we set $a_{k+1,k}=0$ and split  the computation into two smaller eigenproblems. 
Here {\tt u} denotes the machine precision. Another kind of deflation can happen 
in the matrix $B$ and it is related to the occurrence of infinite eigenvalues.  If $b_{k,k}$ is numerically zero then 
there exists at least an infinite eigenvalue and this can be deflated by moving up the zero entry to the  top left corner of $B$.
The criterion used in our implementation to check the nullity of  $b_{k,k}$  is 
\[
|b_{k,k}|\leq {\tt u}\parallel B\parallel.
\]
Eligible shift polynomials are generally determined from the (generalized) eigenvalues of the trailing principal submatrices of $A$ and $B$
We  first compute the generalized eigenvalues  $(\alpha_1, \beta_1)$ and $(\alpha_2, \beta_2)$
of the matrix pair $(A(n-1\colon n, n-1\colon n), B(n-1\colon n, n-1\colon n))$. If they correspond with  a pair of 
  complex conjugate numbers then we 
set 
\[
p(z)=(\beta_2 z -\alpha_2) (\beta_2 z -\alpha_2).
\]
Otherwise we perform a linear shift, that is, $p(z)=\beta z-\alpha$, where  the eigenvalue 
$\sigma=\alpha/\beta$ is the closest to the  value $a_{N,N}/b_{N,N}$.

Our  resulting  algorithm has been tested on several numerical examples. We begin with some classical polynomials that are meant to test the algorithm for speed and for backward stability. With the exception of Example \ref{ex:polynorm}, all polynomials are normalized so as to have 2-norm equal to 1: in practice, the algorithm is always applied to $p/\|p\|_2$. Absolute forward and backward errors for a polynomial $p(x)=\sum_{j=0}^N p_jx^j=p_N\prod_{k=1}^N(x-\alpha_k)$ are defined as
\begin{eqnarray*}
{\rm forward\, error}=\max_{k=1,\ldots ,N}|\alpha_k-\tilde{\alpha}_k|,\\
{\rm backward\, error}=\max_{j=0,\ldots ,N}|p_j-\tilde{p}_j|,
\end{eqnarray*}
where $\{\tilde{\alpha}_k\}_{k=1,\ldots,N}$ are the computed roots, and $\{\tilde{p}_j\}_{j=0,\ldots,N}$ are the polynomial coefficients reconstructed from the computed roots, working in high precision. The polynomial $\tilde{p}(x)=\sum_{j=0}^N \tilde{p}_jx^j$ is also normalized so that $\|\tilde{p}\|_2=1$ prior to backward error computation.

Examples \ref{example:fortran_random} and \ref{example:fortran_cyclo} use the Fortran implementation of Fast QZ, compiled with GNU Fortran compiler and running under Linux Ubuntu 14.04 on a laptop equipped with an Inter i5-2430M processor and 3.8 GB memory. 
All the other tests are based on the MATLAB version of the code and were run on a Mac Book Pro equipped with MATLAB R2016a.

\begin{example}\label{example:fortran_random}
{\em Fortran implementation applied to random polynomials.} Polynomial coefficients are random real numbers uniformly chosen in $[-1,1]$. Here $N$ denotes the degree. Table \ref{tab:1random} shows forward  absolute errors w.r.t. the roots computed by LAPACK, as well as the average number of iterations per eigenvalue and the running times, in seconds, for LAPACK and Fast QZ. All the results are averages over 10 runs for each degree.

In this example, Fast QZ is faster than LAPACK for polynomials of degree larger than $250$. (Of course, the results of timing comparisons may vary slightly depending on the machine and architecture). The quadratic growth of the running time for our algorithm is shown in Figure \ref{figure_random}.
\end{example}
\begin{table}
\caption{Timings and errors for the Fortran implementation of Fast QZ applied to random polynomials.}\label{tab:1random}
\centering
\begin{tabular}{c|c|c|c|c}
\hline\noalign{\smallskip}
$N$  & abs. forward error & average n. it.&Fast QZ time&LAPACK time\\
\noalign{\smallskip}\hline\noalign{\smallskip}
$50$&$1.34$e$-14$&$1.82$&$1.19$e$-2$&$8.80$e$-3$\\
$100$&$1.09$e$-14$&$1.67$&$2.73$e$-2$&$9.70$e$-3$\\
$200$&$1.87$e$-14$&$1.59$&$8.84$e$-2$&$6.26$e$-2$\\
$300$&$3.03$e$-14$&$1.50$&$1.76$e$-1$&$1.97$e$-1$\\
$400$&$1.88$e$-13$&$1.46$&$3.12$e$-1$&$4.71$e$-1$\\
$500$&$8.08$e$-14$&$1.42$&$4.72$e$-1$&$1.18$\\
$600$&$4.73$e$-13$&$1.45$&$7.03$e$-1$&$2.32$\\
$700$&$2.19$e$-13$&$1.41$&$9.54$e$-1$&$4.04$\\
$800$&$1.46$e$-13$&$1.39$&$1.22$&$5.15$\\
$900$&$1.04$e$-13$&$1.37$&$1.50$&$9.00$\\
$1000$&$1.57$e$-13$&$1.39$&$1.90$&$13.06$\\
\noalign{\smallskip}\hline
\end{tabular}
\end{table}

\begin{figure}
\centering
  \includegraphics[width=\textwidth]{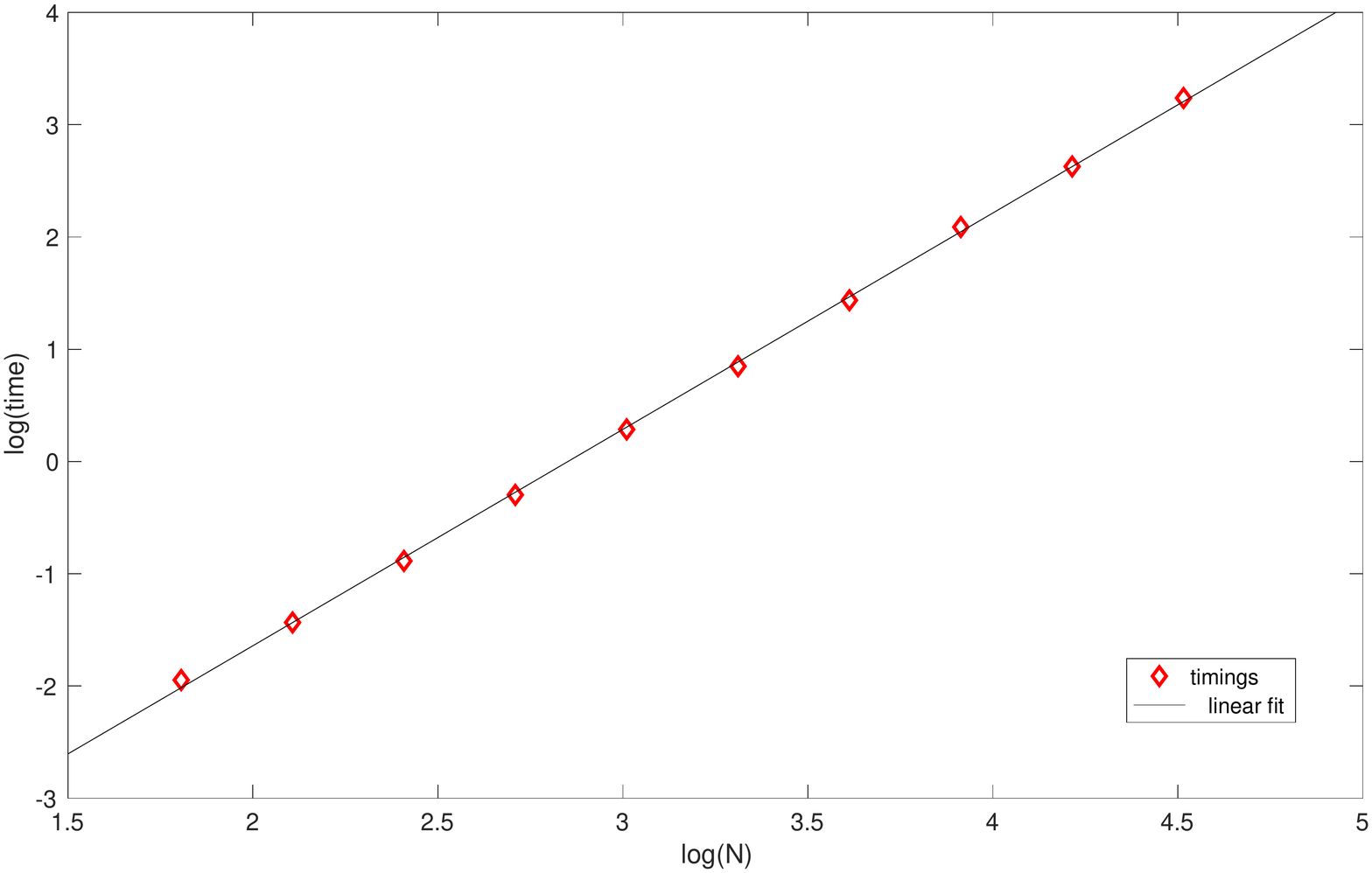}
\caption{This is a log-log plot of running times vs. polynomial degree $N$ for Example \ref{example:fortran_random}. Here we have chosen $N$ as powers of $2$, from $N=2^6=64$ to $N=2^{15}=32768$. The linear fit has equation $y=1.93x-5.50$, which is consistent with the $\mathcal{O}(N^2)$ complexity of Fast QZ.}
\label{figure_random}      
\end{figure}

\begin{example}\label{example:fortran_cyclo}
{\em Fortran implementation applied to cyclotomic polynomials.} The polynomials used in this example take the form $p(x)=x^N-1$. In this case we know the exact roots, which can be computed using the Fortran function {\tt cos} and {\tt sin}. We can therefore compute errors for Fast QZ and for Lapack, both with respect to the ``exact'' roots: FastQZ turns out to be as accurate as LAPACK. Table \ref{tab:1cyclo} shows forward  absolute errors, as well as the average number of iterations per eigenvalue and running times (in seconds). 
Figure \ref{figure_cyclo} shows a logarithmic plot of the running times for Fast QZ, together with a linear fit.
\end{example}

\begin{table}
\caption{Timings and absolute forward errors for the Fortran implementation of Fast QZ applied to cyclotomic polynomials. Errors are computed w.r.t. ``exact'' roots.}\label{tab:1cyclo}
\centering
\begin{tabular}{c|c|c|c|c|c}
\hline\noalign{\smallskip}
$N$  & err. Fast QZ & err. LAPACK & average n. it.&Fast QZ time&LAPACK time\\
\noalign{\smallskip}\hline\noalign{\smallskip}
$100$&$4.65$e$-15$&$3.11$e$-15$&$1.38$&$3.00$e$-2$&$9.00$e$-3$\\
$200$&$5.31$e$-15$&$8.67$e$-15$&$1.25$&$7.90$e$-2$&$5.20$e$-2$\\
$300$&$6.76$e$-15$&$1.37$e$-14$&$1.19$&$1.52$e$-1$&$1.67$e$-1$\\
$400$&$1.05$e$-14$&$1.74$e$-14$&$1.16$&$2.82$e$-1$&$3.97$e$-1$\\
$500$&$9.49$e$-15$&$2.28$e$-14$&$1.14$&$4.03$e$-1$&$1.03$\\
$600$&$1.46$e$-14$&$2.85$e$-14$&$1.12$&$5.79$e$-1$&$2.01$\\
$700$&$1.51$e$-14$&$3.19$e$-14$&$1.12$&$7.85$e$-1$&$3.40$\\
$800$&$1.53$e$-14$&$3.90$e$-14$&$1.10$&$9.73$e$-1$&$5.39$\\
$900$&$1.93$e$-14$&$3.95$e$-14$&$1.10$&$1.24$&$8.00$\\
$1000$&$1.69$e$-14$&$4.84$e$-14$&$1.10$&$1.53$&$11.18$\\
$1500$&$3.00$e$-14$&$7.37$e$-14$&$1.09$&$3.35$&$41.47$\\
$2000$&$2.45$e$-14$&$1.02$e$-13$&$1.08$&$5.80$&$107.97$\\
\noalign{\smallskip}\hline
\end{tabular}
\end{table}

\begin{figure}
\centering
  \includegraphics[width=\textwidth]{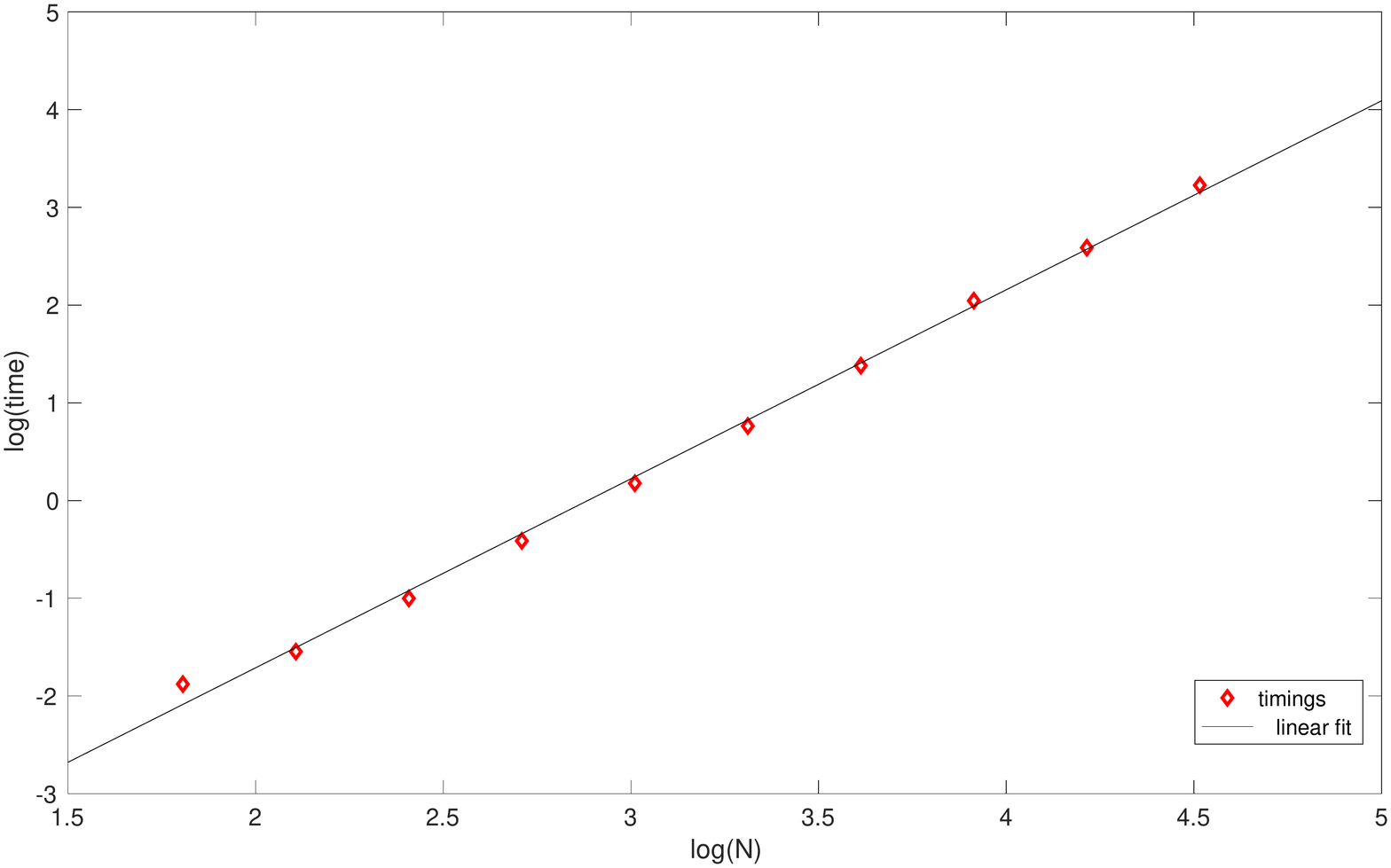}
\caption{Log-log plot of running times vs. polynomial degree for Example \ref{example:fortran_cyclo}. The linear fit has equation $y=1.93x-5.58$.}
\label{figure_cyclo}      
\end{figure}

\begin{example}
In this example we use a classical set of test polynomials taken from \cite{TT} . The polynomials are all of degree 20:
\begin{enumerate}
\item the Wilkinson polynomial, i.e., $P(x)=\prod_{k=1}^{20}(x-k)$,
\item
the polynomial with roots uniformly spaced in [-1.9, 1.9],
\item
$P(x)=\sum_{k=0}^{20}x^k/k!$,
\item
the Bernoulli polynomial of degree 20,
\item
$P(x)=1+x+x^2+\ldots+x^{20}$,
\item
the polynomial with roots $2^{-10}, 2^{-9},\ldots ,2^9$,
\item
the Chebyshev polynomial of degree 20.
\end{enumerate}
Table \ref{tab:2} shows absolute forward and backward errors for our fast QZ and for classical QZ applied to the companion pencil. For the purpose of computing forward errors we have taken as $\{\alpha_k \}_{k=1,\ldots,N}$ either the exact roots, if known, or numerical roots computed with high accuracy.

Forward errors may vary, consistently with the conditioning of the problem. However, backward errors are always of the order of the machine epsilon, which points to a backward stable behavior in practice.

\end{example}

\begin{table}
\caption{Forward and backward errors for a set of ill-conditioned polynomials. Note that the MATLAB implementation of classical QZ sometimes finds infinite roots, which prevent computation of the backward error. This behavior is denoted by the entry Inf.}
\label{tab:2}  
\centering
\begin{tabular}{c|c|c|c|c}
\hline\noalign{\smallskip}
$P(x)$  & f. err. (fast QZ) & f. err. (classical QZ) & b. err. (fast QZ) & b. err. (classical QZ)\\
\noalign{\smallskip}\hline\noalign{\smallskip}
1&$28.73$&Inf&$6.52$e$-16$&Inf\\
2&$5.91$e$-13$&$8.07$e$-13$&$8.07$e$-16$&$1.11$e$-15$\\
3&$5.70$&Inf&$2.22$e$-16$&Inf\\
4&$3.76$e$-10$&$1.83$e$-12$&$1.72$e$-15$&$1.20$e$-15$\\
5&$3.06$e$-15$&$1.09$e$-15$&$4.52$e$-15$&$1.58$e$-15$\\
6&$1.09$e$-2$&$2.30$e$-3$&$2.28$e$-15$&$3.05$e$-15$\\
7&$5.47$e$-11$&$1.68$e$-11$&$1.08$e$-15$&$1.91$e$-15$\\
\noalign{\smallskip}\hline
\end{tabular}
\end{table}

\begin{example}\label{example:JT}
We apply here our structured algorithm to some polynomials taken from the test suite proposed by Jenkins and Traub in \cite{JT}.
The polynomials are:
\begin{itemize}
\item[] $p_1(x)=((x-a)(x-1)(x+a))$, with $a= 10^{-8}$, $10^{-15}$, $10^8$, $10^{15}$,
\item[] $p_3(x)=\prod_{j=1}^r(x-10^{-j})$, with $r=10$, $20$,
\item[] $p_4(x)=(x-0.1)^3(x-0.5)(x-0.6)(x-0.7)$,
\item[] $p_7=(x-0.001)(x-0.01)(x-0.1)(x-0.1+ai)(x-0.1-ai)(x-1)(x-10)$, with $a=10^{-10}$,
\item[] $p_{10}(x)=(x-a)(x-1)(x-a^{-1})$, with $a=10^3$, $10^6$, $10^9$,
\item[] $p_{11}(x)=\prod_{j=1-m}^{m-1}(x-e^{\frac{ij\pi}{2m}})\prod_{j=m}^{3m}0.9e^{\frac{ij\pi}{2m}}$, with $m=15$.
\end{itemize}
In particular, the polynomial $p_1(x)$ is meant to test whether large or 
small zeros may pose a difficulty, the polynomial $p_3(x)$ can be used to test for underflow, the polynomials $p_4(x)$ and $p_7(x)$ test for multiple or nearly multiple roots, whereas $p_{10}(x)$ and $p_{11}(x)$ test for deflation stability. Table \ref{tab:2.5} shows absolute forward and backward errors, computed as in the previous example, for Fast QZ and classical QZ. Note that larger values of $r$ for $p_3(x)$ tend to slow down convergence, so for $r=20$ we needed to increase the allowed number of iterations per eigenvalue (before an exceptional shift is applied). 

When QZ is tested on the polynomial $p_1(x)$ with large values of $a$, normalization of the coefficients inevitably leads to a numerically zero leading coefficient and therefore to infinite eigenvalues. In this case, both fast and classical QZ retrieve the root $1$ with accuracy up to machine precision. Of course one may also try using the non-normalized polynomials, in which case Fast QZ finds roots $\{1, 1, -1\}$ and classical QZ finds roots $\{1, 0, 0\}$, up to machine precision.
\end{example}

\begin{table}
\caption{Forward and backward errors for polynomials taken from Jenkins and Traub's test suite; see Example \ref{example:JT}.}
\label{tab:2.5}  
\centering
\begin{tabular}{l|c|c|c|c}
\hline\noalign{\smallskip}
$P(x)$  & f. err. (fast QZ) & f. err. (class. QZ) & b. err. (fast QZ) & b. err. (class. QZ)\\
\noalign{\smallskip}\hline\noalign{\smallskip}
$p_1(x)$, $a=1$e$-8$&$1.52$e$-8$&$1.00$e$-8$&$2.22$e$-16$&$1.11$e$-16$\\
$p_1(x)$, $a=1$e$-15$&$1.64$e$-8$&$8.08$e$-16$&$1.90$e$-16$&$1.11$e$-16$\\
$p_3(x)$, $r=10$&$8.76$e$-6$&$1.00$e$-6$&$8.60$e$-16$&$3.61$e$-16$\\
$p_3(x)$, $r=15$&$1.25$e$-6$&$1.37$e$-6$&$6.80$e$-16$&$9.10$e$-16$\\
$p_3(x)$, $r=20$&$1.99$e$-4$&$9.90$e$-7$&$3.14$e$-15$&$8.07$e$-16$\\
$p_4(x)$ & $9.088$e$-6$&$4.26$e$-6$&$6.66$e$-16$&$3.33$e$-16$\\
$p_7(x)$, $a=1$e$-10$&$1.91$e$-5$&$6.47$e$-6$&$2.77$e$-16$&$1.11$e$-16$\\
$p_{10}(x)$, $a=1$e$+3$&$2.71$e$-16$&$0$&$1.91$e$-16$&$0$\\
$p_{10}(x)$, $a=1$e$+6$&$1.16$e$-16$&$8.25$e$-18$&$8.20$e$-17$&$5.83$e$-18$\\
$p_{10}(x)$, $a=1$e$+9$&$1.81$e$-16$&$0$&$1.28$e$-16$&$1.11$e$-16$\\
$p_{11}(x)$, $m=15$&$1.11$e$-14$&$9.87$e$-15$&$3.45$e$-14$&$1.80$e$-14$\\
\noalign{\smallskip}\hline
\end{tabular}
\end{table}

\begin{example}\label{example:jumping}
{\em The jumping polynomial.} This is a polynomial of degree 20 where the coefficients are heavily unbalanced and QZ applied to the companion pencil tends to work better than computing the eigenvalues of the companion matrix (see also \cite{last}). The polynomial is defined as $p(x)=\sum_{k=0}^{20} p_kx^k$, where  $p_k=10^{6(-1)^{(k+1)} -3}$ for $k=0,\ldots,20$. Table \ref{tab:3} shows that Fast QZ is just as accurate as classical QZ, and more accurate than the MATLAB command {\tt roots}.
\end{example}

\begin{table}
\caption{Forward and backward errors for several methods applied to a polynomial with highly unbalanced coefficients (Example \ref{example:jumping}).}
\label{tab:3} 
\centering
\begin{tabular}{c|c|c}
\hline\noalign{\smallskip}
method  & forward error & backward error \\
\noalign{\smallskip}\hline\noalign{\smallskip}
fast QZ&$2.78$e$-15$&$4.94$e$-15$\\
classical QZ&$1.49$e$-15$&$3.22$e$-15$\\
balanced QR&$2.46$e$-9$&$5.86$e$-9$\\
unbalanced QR&$1.68$e$-15$&$2.72$e$-15$\\
\noalign{\smallskip}\hline
\end{tabular}
\end{table}

\begin{example}\label{ex:polynorm}
In order to test the behavior of backward error for non-normalized polynomials (that is, for unbalanced pencils), we consider polynomials of degree $50$ with random coefficients and 2-norms ranging from $1$ to $10^{14}$. For each polynomial $p$ we apply QZ (structured or unstructured) without normalization to compute its roots. Then we form a polynomial $\tilde{p}$ from the computed roots, working in high precision, and define the 2-norm absolute backward error as 
$$
{\rm backward\, error_2}=\min_{\alpha\in\mathbb{R}}\|p-\alpha\tilde{p}\|_2.
$$ 
In practice, the value of $\alpha$ that minimizes the backward error is computed as $\alpha=\left(\sum_{i=0}^N p_i\tilde{p}_i\right) /\sum_{i=0}^N\tilde{p}_i^2$.

Figure \ref{fig:polynorm} shows that in this example the backward error grows proportionally to $\|p\|_2^2$ and its behavior when using Fast QZ is very similar to the case of classical QZ (that is, the Matlab function {\tt eig}). See also the analysis in \cite{AMRVW}.
\end{example}

\begin{figure}
\centering
  \includegraphics[width=0.8\textwidth]{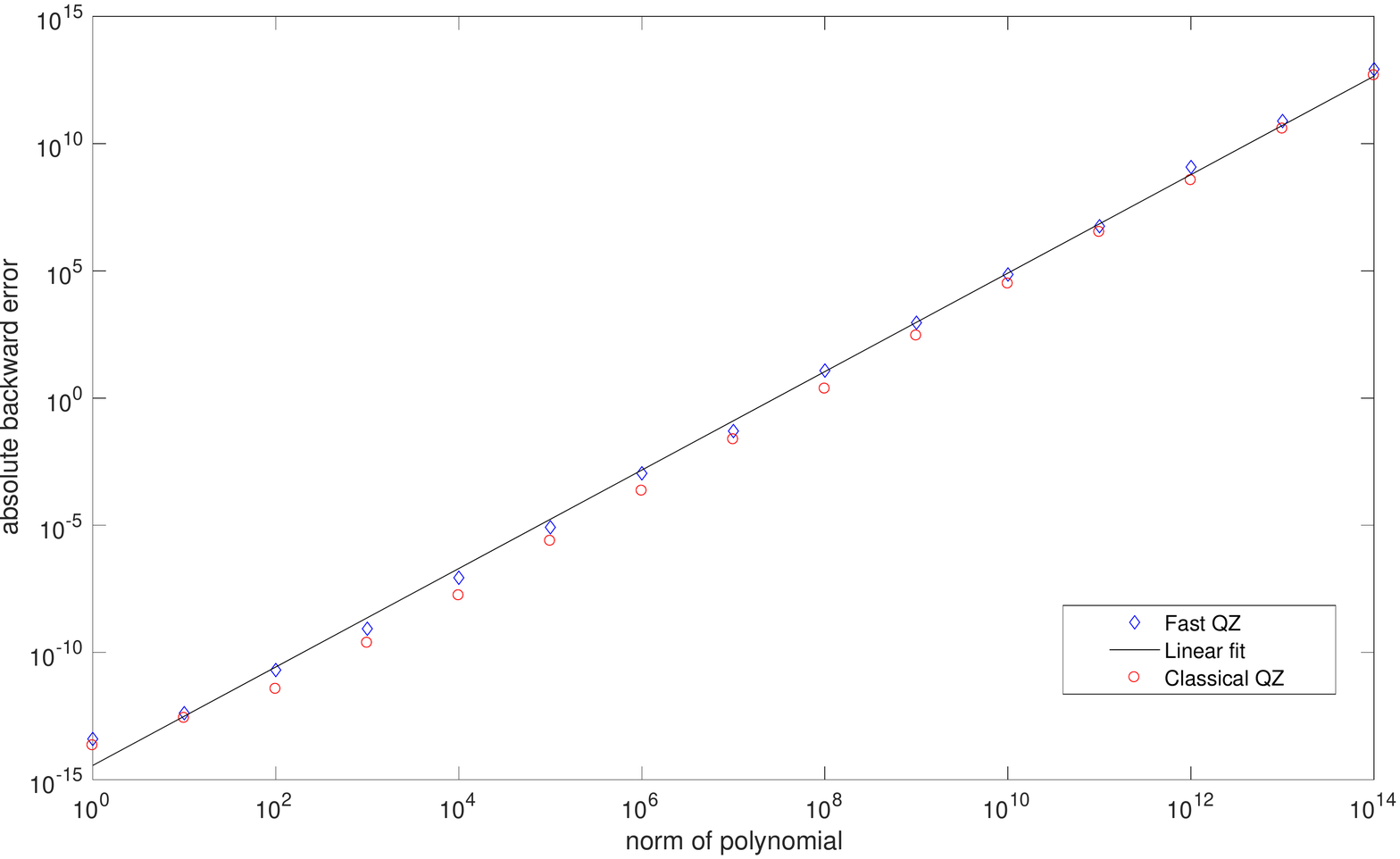}
\caption{Absolute backward error vs. polynomial norm for Example \ref{ex:polynorm}. The black line is a linear fit for the backward error of Fast QZ. Its equation is $y=1.94x-14.4$, which suggests that the absolute backward error grows proportionally to $\|p\|_2^2$.}
\label{fig:polynorm}      
\end{figure}

Our  algorithm has been tested on several numerical examples resulting from the linearization of nonlinear eigenvalue problems by using 
Lagrange type interpolation schemes. In particular, if $f\ \Omega\subseteq  \mathbb R \ \rightarrow \ \mathbb R$  is analytic then 
increasingly accurate approximations 
of its  zeros can be found  by  rootfinding  methods applied to certain   polynomial approximations of $f$.  
 The  unique polynomial of degree less than $n$ interpolating the function $f(z)$  at the $N-$th roots 
of unity $z_k=e^{\displaystyle{2\pi(k-1)/N}}$, $1\leq k\leq N$, can be expressed as 
\[
p(z)=(z^N-1)\sum_{j=1}^n\frac{w_j f_j}{z-z_j}, 
\]
where 
\[
f_j=f(z_j), \quad w_j=\Big( \prod_{k=1,k\neq j}(z_j-z_k)\Big)^{-1}=  z_j/N, \quad 1\leq j\leq N.
\]
In \cite{Cor} it was shown that the roots of $p(z)$ are the finite eigenvalues of the matrix pencil $F-z G$, 
$F,G\in \mathbb C^{(N+1)\times (N+1)}$, given by 
\begin{equation}\label{f1}
F=\left[\begin{array}{cccc}
0 & -f_1/\xi_1& \ldots & -f_N/\xi_N\\
w_1 \xi_1& z_1 \\
\vdots & & \ddots\\
w_N\xi_N & & &z_N
\end{array}\right], \quad 
G=\left[\begin{array}{cccc}
0 \\
 & 1 \\
 & & \ddots\\
 & & & 1
\end{array}\right], 
\end{equation}
where $\xi_1, \ldots, \xi_N$ are nonzero  real numbers used  for balancing purposes.
Observe that since the size of the matrices is $N+1$ we obtain at least two spurious infinite eigenvalues. 
By  a suitable congruence transformation $F\rightarrow F_1=Q F Q^*$ and 
$G\rightarrow G_1=Q G Q^*$ with $Q$ orthogonal,  we generate  an equivalent   real matrix pair $(F_1, G_1)$  where 
$G_1=G$ and $F_1$ is arrowhead  with $2\times 2$  orthogonal diagonal blocks. Then the  usual Hessenberg/triangular reduction procedure 
can be applied by returning   a final  real matrix pair  $(\tilde A,\tilde B)$.  One infinite 
eigenvalue can immediately be deflated  by simply performing a  permutation between the first and the second rows of $\tilde A$ and $\tilde B$ by 
 returning a final matrix pair $(A,B)$ belonging to the class $\mathcal P_{N}$.  It can be shown that if 
$\xi_i=\xi$ for all $i$ then this latter matrix  pair is the companion pencil associated with the interpolating polynomial expressed in the power basis. 
Otherwise, if $\xi_i$ are not constant then $A$ is generally a dense Hessenberg matrix which can be represented as a rank one 
modification of an orthogonal matrix.

In the following examples we test the application of Fast QZ to the barycentric Lagrange interpolation on the roots of unity in order to find zeros of functions  or solve eigenvalue problems. We point out here some implementation details:
\begin{itemize}
\item Scaling. The first row and column of the matrix $F$ can be scaled independently without modifying $G$. We consistently normalize them so that $\|F(1,:)\|_2=$\\$\|F(:,1)\|_2=1$, which makes the pencil more balanced.
\item Deflation of infinite eigenvalues. Spurious infinite eigenvalues can be eliminated by applying repeatedly the permutation trick outlined above for the pencil $(\tilde{A},\tilde{B})$. This leaves us, of course, with the problem of choosing a suitable deflation criterion. In practice, we perform this form of deflation when $|\tilde{A}(1,1)|<\varepsilon\sqrt{N}$, where $\varepsilon$ is the machine epsilon. 
\item Reduction of the arrowhead pencil to Hessenberg/triangular form: this can be done in a fast (e.g., $O(N^2)$) way via Givens rotations that exploit structure, see e.g. \cite{Law_thesis}, Section 2.2.2.
\end{itemize}

\begin{example}\label{ex:AKT1}
 This example is discussed in \cite{AKT}. Consider the function $f(z)=\sin(z-0.3)\log(1.2-z)$. We seek the zeros of $f$ in the unit disk; the exact zeros are $0.2$ and $0.3$. Table \ref{tab:AKT1} shows the computed approximations of these zeros for several values of $N$ (number of interpolation points). The results are consistent with findings in \cite{AKT}, where $50$ interpolation points yielded an accuracy of 4 digits.
\end{example}

\begin{table}
\caption{Approximations of the zeros of $f(z)=\sin(z-0.3)\log(1.2-z)$. Here $N$ is the number of interpolation points. See Example \ref{ex:AKT1}.}
\label{tab:AKT1} 
\centering
\begin{tabular}{c|l|l}
\hline\noalign{\smallskip}
$N$  & approx. of $0.2$ & approx. of $0.3$ \\
\noalign{\smallskip}\hline\noalign{\smallskip}
$20$&$0.2153$ &$0.2841$\\
$30$&$0.2014$ &$0.2986$\\
$40$&$0.20016$&$0.29983$\\
$50$&$0.200021$&$0.299978$\\
$60$&$0.2000028$&$0.2999970$\\
$100$&$0.2000000011$&$0.2999999988$\\
$200$&$0.199999999999894$&$0.300000000000120$\\
\noalign{\smallskip}\hline
\end{tabular}
\end{table}

\begin{example}\label{ex:matrixeig}
This is also an example from \cite{AKT}. Define the matrix
$$
A=\left(\begin{array}{rrrr}
3.2&1.5&0.5&-0.5\\
-1.6&0.0&-0.4&0.6\\
-2.1&-2.2&0.2&-0.1\\
20.7&9.3&3.9&-3.4\\
\end{array}\right).
$$
We want to compute its eigenvalues by approximating the zeros of the polynomial $p(\lambda)=\det({A}-\lambda I)$. The exact eigenvalues are $0.2$, $0.3$, $1.5$ and $-2$. 
Interpolation plus Fast QZ using $6$ nodes yields all the correct eigenvalues up to machine precision.

One may also apply a similar approach to the computation of the eigenvalues in the unit circle for a larger matrix. See Figures \ref{fig:randmatrix120} and \ref{fig:randmatrix60} for tests on two $100\times 100$ matrices with random entries (uniformly chosen in [-1,1]). \end{example}

\begin{figure}
\centering
  \includegraphics[width=\textwidth]{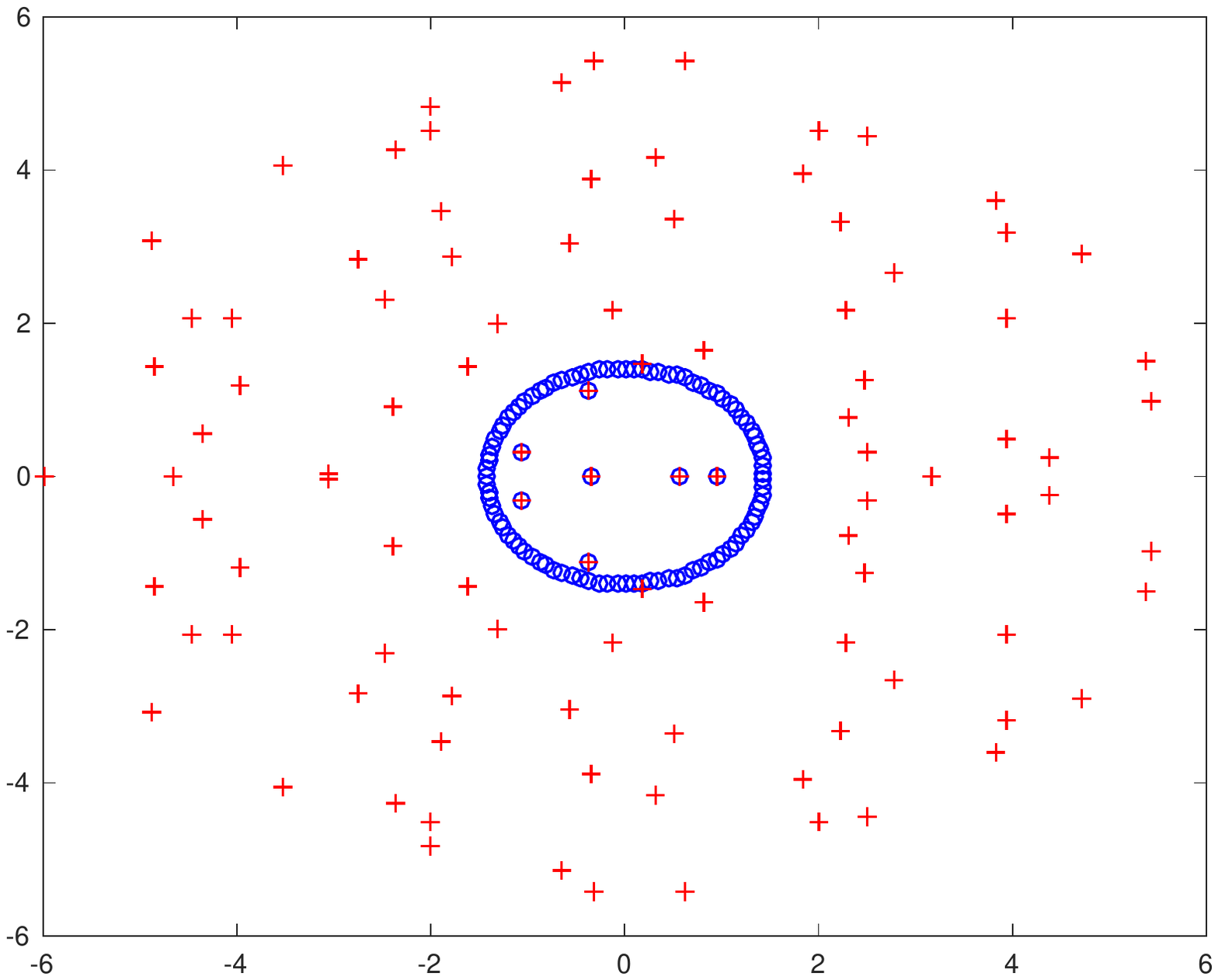}
\caption{Eigenvalues of a $100\times 100$ random matrix; see Example \ref{ex:matrixeig}. The blue circles are the eigenvalues computed via interpolation, the red crosses are the eigenvalues computed by {\tt eig}. Here 120 interpolation nodes were used.}
\label{fig:randmatrix120}      
\end{figure}
\begin{figure}
  \includegraphics[width=\textwidth]{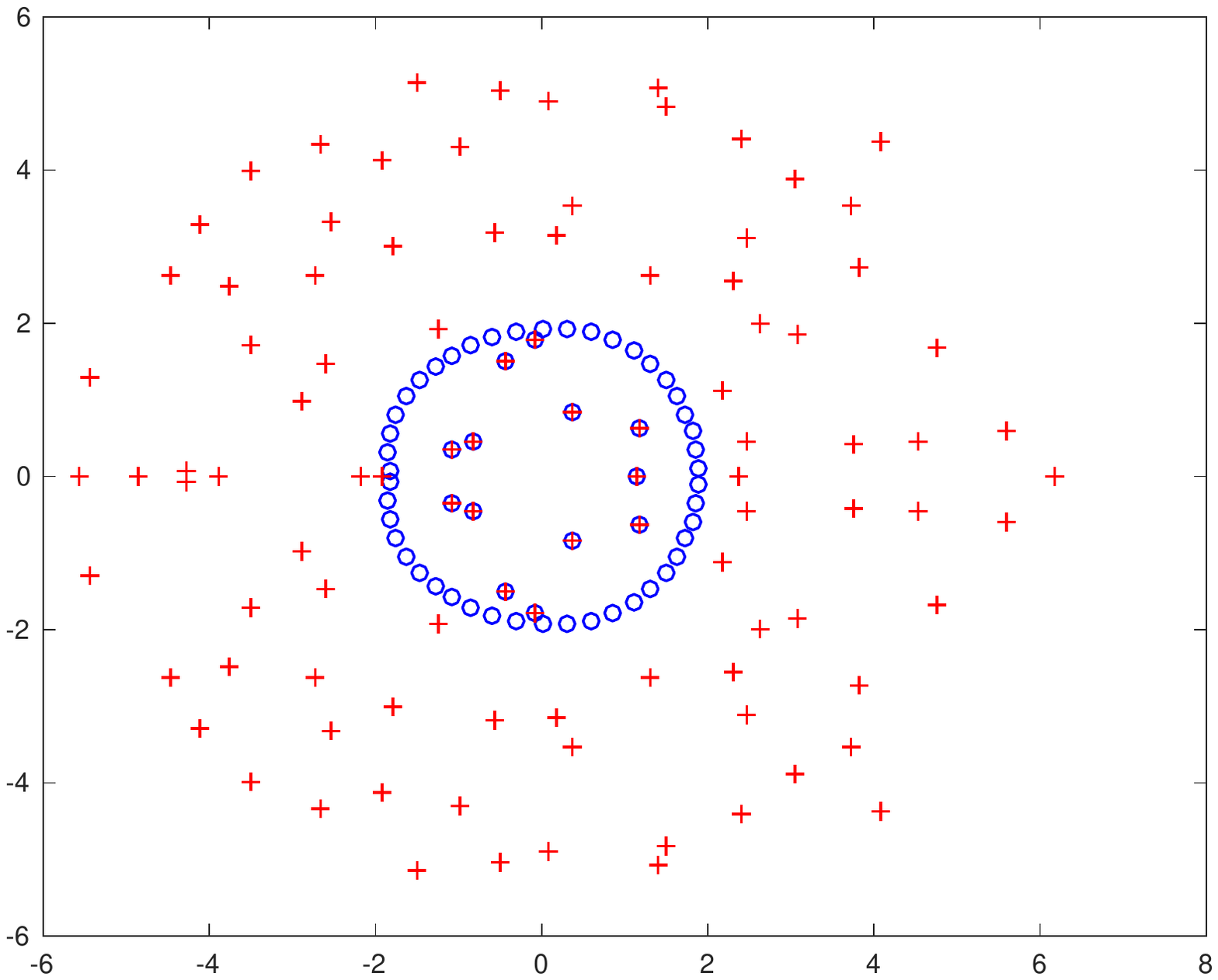}
\caption{Eigenvalues of a $100\times 100$ random matrix; see Example \ref{ex:matrixeig}. The blue circles are the eigenvalues computed via interpolation, the red crosses are the eigenvalues computed by {\tt eig}. Here 60 interpolation nodes were used.}
\label{fig:randmatrix60}      
\end{figure}

\begin{example}\label{ex:nlevp}
We consider some nonlinear eigenvalue problems taken from \cite{NLEVP}:
\begin{enumerate}
\item {\tt mobile\_manipulator}: this $5\times 5$ quadratic matrix polynomial is close to being nonregular;  
\item {\tt gen\_tpal2}: a real T-palindromic quadratic matrix polynomial of size $16\times 16$ whose eigenvalues lie on the unit circle;
\item {\tt closed\_loop}: the eigenvalues of this $2\times 2$ parameterized quadratic polynomial lie inside the unit disc for a suitable choice of the parameter;
\item {\tt relative\_pose\_5pt}: a $10\times 10$ cubic matrix polynomial which comes from the five point relative pose problem in computer vision. See Figure \ref{fig:pose5pt} for a plot of the eigenvalues.
\end{enumerate}
Table \ref{tab:nlevp} shows the distance, in $\infty$-norm, between the eigenvalues computed via interpolation followed by Fast QZ and the eigenvalues computed via {\tt polyeig}. 
\end{example}

\begin{table}
\caption{Distance between the eigenvalues computed by interpolation+Fast QZ and {\tt polyeig}, for some problems taken from the NLEVP suite (Example \ref{ex:nlevp}). Here $N$ is the number of interpolation points. The error for the fourth problem is computed on all the eigenvalues (fourth line) and on the eigenvalues in the unit disk (fifth line, error marked by an asterisk.)}
\label{tab:nlevp} 
\centering
\begin{tabular}{c|c|c}
\hline\noalign{\smallskip}
problem  & error & $N$ \\
\noalign{\smallskip}\hline\noalign{\smallskip}
{\tt mobile\_manipulator}&$2.53$e$-15$ &$20$\\
{\tt gen\_tpal2}&$1.61$e$-9$ &$50$\\
{\tt closed\_loop}&$1.22$e$-15$&$10$\\
{\tt relative\_pose\_5pt}&$2.81$e$-10$&$40$\\
{\tt relative\_pose\_5pt}&$8.99$e$-15^{(*)}$&$40$\\
\noalign{\smallskip}\hline
\end{tabular}
\end{table}

\begin{figure}
\centering
  \includegraphics[width=\textwidth]{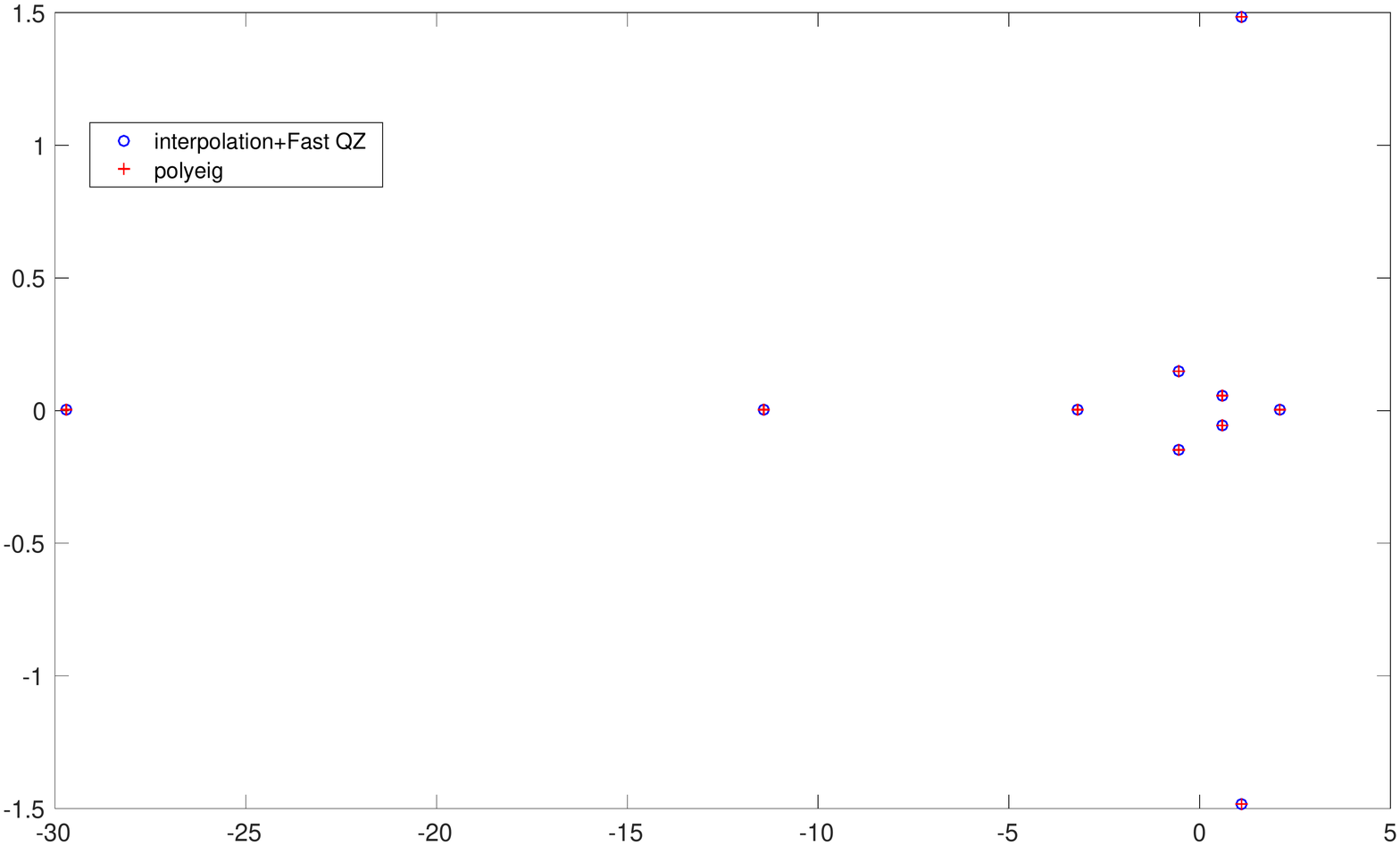}
\caption{Eigenvalues of the matrix polynomial {\tt relative\_pose\_5pt}; see Example \ref{ex:nlevp}. The blue circles are the eigenvalues computed via interpolation, the red crosses are the eiganvalues computed by {\tt polyeig}.}
\label{fig:pose5pt}      
\end{figure}

\begin{example}\label{ex:random} 
Random matrix polynomials: we use matrix polynomials with random coefficients (given by the Matlab function {\tt rand}). Table \ref{tab:random} shows errors with respect to {\tt polyeig} for several values of the degree and of the size of the polynomial.
\end{example}
\begin{table}
\caption{Distance between the eigenvalues computed by interpolation+Fast QZ and {\tt polyeig}, for random matrix polynomials of different degrees and sizes (Example \ref{ex:random}). The error is computed on the eigenvalues contained in the disk of center $0$ and radius $2$.}
\label{tab:random} 
\centering
\begin{tabular}{c|c|c}
\hline\noalign{\smallskip}
degree  & size & error \\
\noalign{\smallskip}\hline\noalign{\smallskip}
$10$&$5$ &$1.12$e$-11$\\
$10$&$10$ &$1.11$e$-9$\\
$10$&$20$ &$2.93$e$-5$\\
$15$&$5$ &$6.98$e$-9$\\
$15$&$10$ &$5.15$e$-9$\\
$15$&$20$ &$3.24$e$-4$\\
$20$&$5$ &$2.13$e$-10$\\
$20$&$10$ &$4.50$e$-9$\\
\noalign{\smallskip}\hline
\end{tabular}
\end{table}

\begin{example}\label{ex:lambert}
We consider here a nonlinear, non polynomial example: the Lambert equation
\begin{equation}
w^6\exp(w^6)=0.1.\label{lambert}
\end{equation}
This equation has two real solutions
$$
w=\pm W(0,0.1)\approx \pm 0.671006
$$
and complex solutions of the form
\begin{eqnarray*}
&& w=\pm\left(W(\nu,x)\right)^{1/6},\\
&& w=\pm(-1)^{1/3}\left(W(\nu,x)\right)^{1/6},\\
&& w=\pm(-1)^{2/3}\left(W(\nu,x)\right)^{1/6},
\end{eqnarray*}
where $W(\nu, x)$, with $\nu\in\mathbb{Z}$ and $x\in\mathbb{C}$, denotes the $\nu$-th branch of the product-log function (Lambert function) applied to $x$.

Such solutions can be computed in Matlab using the {\tt lambertw} function: in the following we will consider them as the ``exact'' solutions. We want to test the behavior of the ``interpolation+QZ'' approach in this case. Experiments suggest the following remarks:
\begin{itemize}
\item
As expected, interpolation only ``catches'' roots inside the unit disk: see Figure \ref{fig2}. Since the roots of \eqref{lambert} are mostly outside the unit disk, we introduce a scaled version of the equation:
\begin{equation}
\alpha^6 w^6\exp(\alpha^6 w^6)=0.1,\label{lambert_scaled}
\end{equation}
where $\alpha\geq 1$ is a scaling parameter. The drawback is that, as $\alpha$ grows, the matrix pencil becomes more unbalanced.
\item
A large number of interpolation nodes is needed (considerably larger than the number of approximated roots). 
\end{itemize} 
See Figure \ref{fig1} for an example.

Tables \ref{alpha1}, \ref{alpha15} and \ref{alpha17} show the accuracy of the approximation for several values of $\alpha$ and of the number of nodes. Here by ``distance'' we denote the distance in $\infty$-norm between the roots of \eqref{lambert_scaled} inside the unit disk and their approximations computed via interpolation followed by structured or unstructured QZ.

\end{example} 

\begin{figure}
\centering
\includegraphics[width=\textwidth]{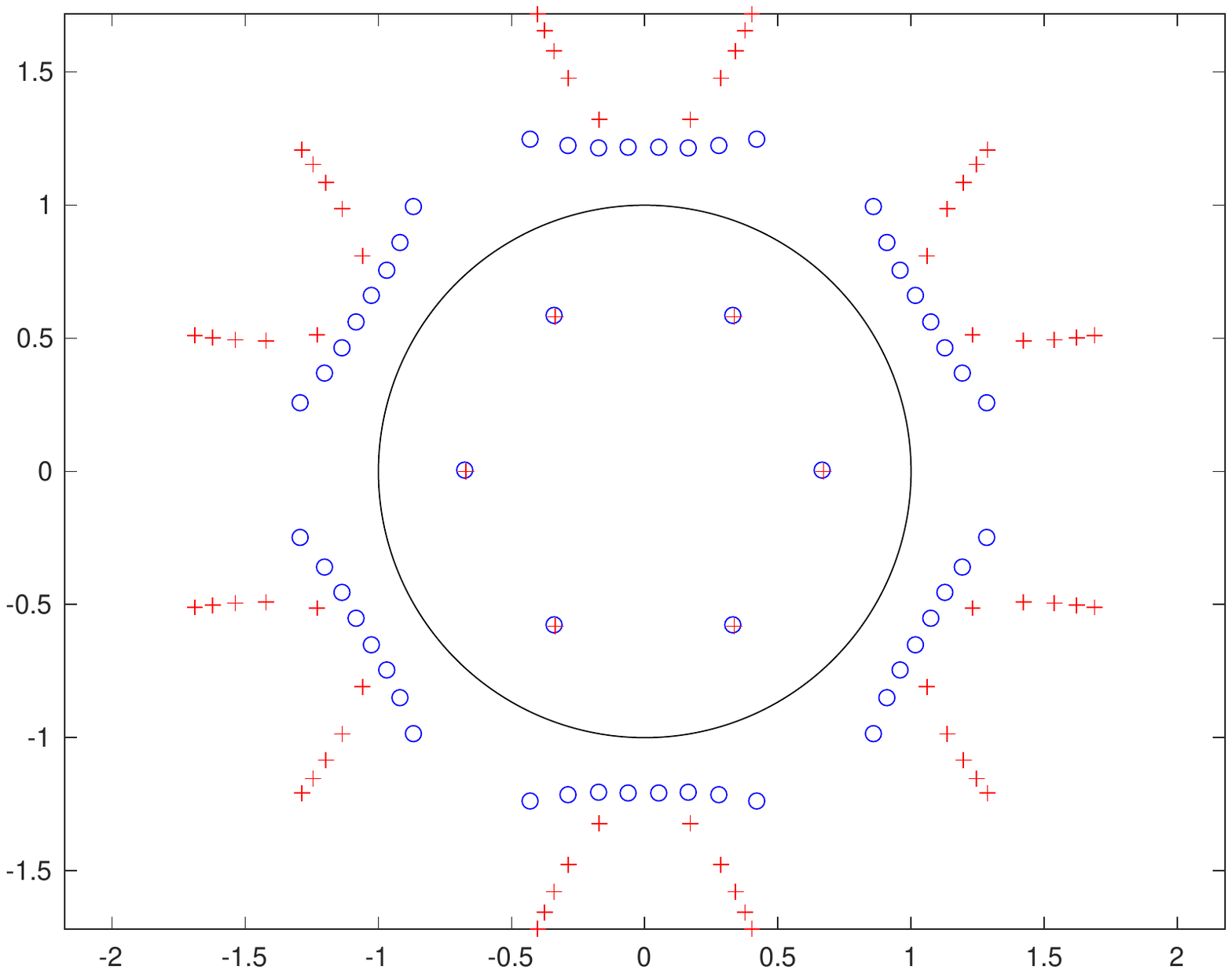}
\caption{This is a plot of the roots of \eqref{lambert} in the complex plane. Red crosses denote the ``exact'' roots computed by {\tt lambertw} with $\nu$ up to $5$. Blue circles are the roots computed via interpolation+QZ with 60 nodes. Note that only the 6 roots inside the unit circle (plotted in black for reference) are correctly approximated.}\label{fig2} 
\end{figure}

\begin{figure}
\centering
\includegraphics[width=0.8\textwidth]{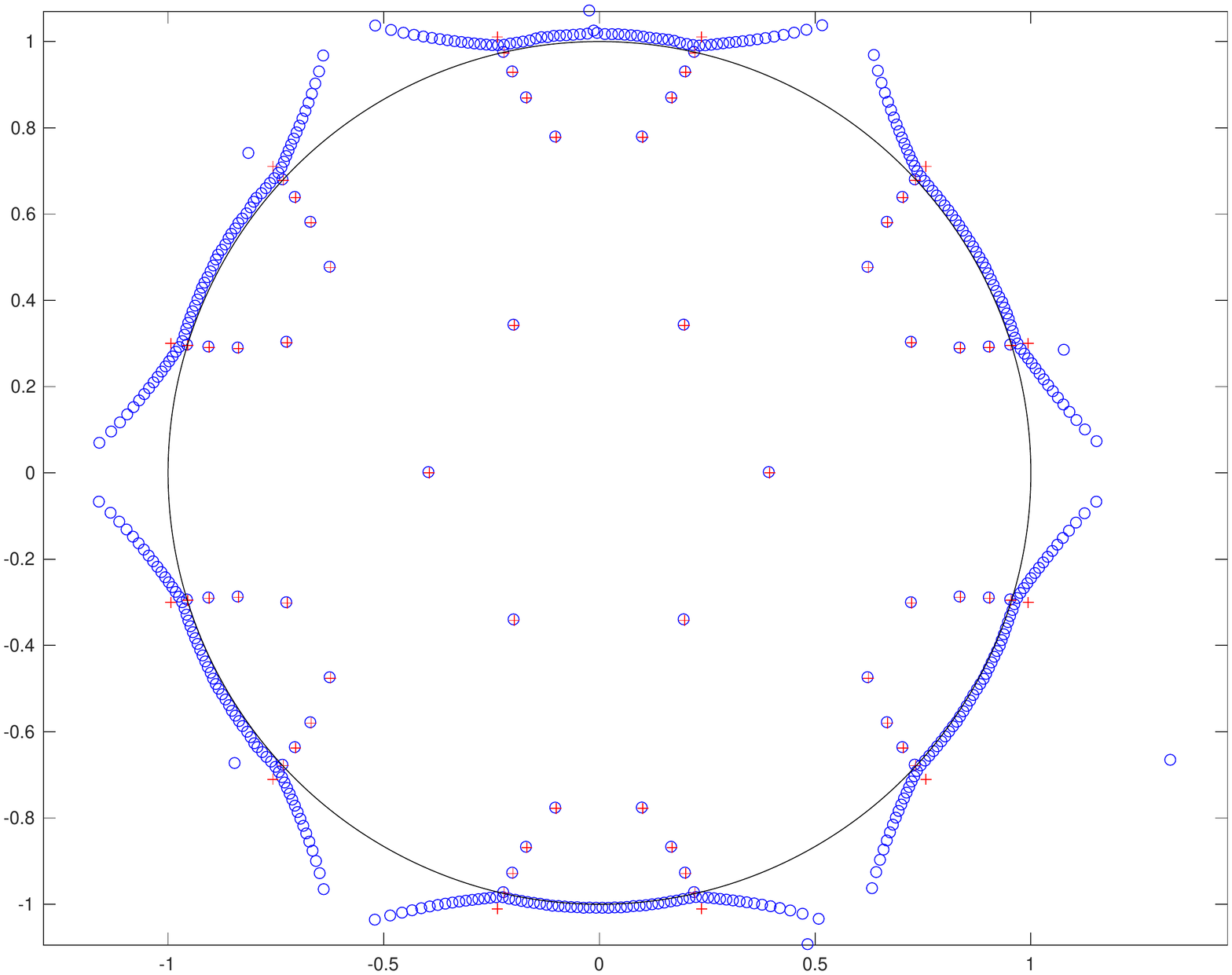}
\caption{This is a plot of the roots of \eqref{lambert_scaled} in the complex plane, with $\alpha=1.7$, so that $54$ roots are inside the unit circle (plotted in black for reference). Red crosses denote the ``exact'' roots computed by {\tt lambertw} with $\nu$ up to $5$. Blue circles are the roots computed via interpolation+QZ. Here we have taken 446 interpolation nodes. }\label{fig1} 
\end{figure}

\begin{table}[h]
\caption{Example \ref{ex:lambert}: we take $\alpha=1$ and there are $6$ roots inside the unit circle. The number of nodes is denoted by $N$, taken as $N=6k+1$ for some $k\in\mathbb{N}$, for symmetry. The accuracy in the approximation of the roots is the same for structured and unstructured QZ. An accuracy of about $10^{-15}$ is reached using 103 nodes.}\label{alpha1}
\centering
\begin{tabular}{|l|l|}
\hline\noalign{\smallskip}
$N$&distance\\
\noalign{\smallskip}\hline\noalign{\smallskip}
$7$&$1.88$e$-1$\\
$19$&$2.53$e$-2$\\
$31$&$1.18$e$-3$\\
$43$&$2.80$e$-5$\\
$55$&$3.88$e$-7$\\
$67$&$3.52$e$-9$\\
\noalign{\smallskip}\hline
\end{tabular}
\end{table}

\begin{table}[h]
\caption{Example \ref{ex:lambert}: we take $\alpha=1.5$ and there are $18$ roots inside the unit circle.}\label{alpha15}
\centering
\begin{tabular}{|l|ll|}
\hline\noalign{\smallskip}
$N$&distance (unstructured)&distance (structured)\\
\noalign{\smallskip}\hline\noalign{\smallskip}
$181$&$1.67$e$-1$&$1.67$e$-1$\\
$217$&$7.50$e$-4$&$7.50$e$-4$\\
$253$&$4.36$e$-7$&$4.36$e$-7$\\
$289$&$7.06$e$-10$&$1.40$e$-9$\\
\noalign{\smallskip}\hline
\end{tabular}
\end{table}

\begin{table}[h]
\caption{Example \ref{ex:lambert}: we take $\alpha=1.6$ and there are $30$ roots inside the unit circle.}\label{alpha16}
\centering
\begin{tabular}{|l|ll|}
\hline\noalign{\smallskip}
$N$&distance (unstructured)&distance (structured)\\
\noalign{\smallskip}\hline\noalign{\smallskip}
$301$&$1.62$e$-3$&$1.62$e$-3$\\
$361$&$9.73$e$-8$&$6.03$e$-7$\\
\noalign{\smallskip}\hline
\end{tabular}
\end{table}

\begin{table}[h]
\caption{Example \ref{ex:lambert}: we take $\alpha=1.7$ and there are $54$ roots inside the unit circle.}\label{alpha17}
\centering
\begin{tabular}{|l|ll|}
\hline\noalign{\smallskip}
$N$&distance (unstructured)&distance (structured)\\
\noalign{\smallskip}\hline\noalign{\smallskip}
$379$&$2.11$e$-1$&$2.11$e$-1$\\
$433$&$4.59$e$-4$&$1.34$e$-3$\\
\noalign{\smallskip}\hline
\end{tabular}
\end{table}

\begin{example}\label{ex:mele}
This example comes from the discretization of a nonlinear eigenvalue problem governing the eigenvibrations of a string with an elastically attached mass: see e.g., \cite{NLEVP}, \cite{effenberger} and \cite{solovev}. The original problem is given by
$$
\left\{\begin{array}{l}
-u''(x)=\lambda u(x)\\
u(0)=0\\
u'(1)+k\frac{\lambda}{\lambda-k/m}u(1)=0\end{array}\right.
$$
where the parameters $k$ and $m$ correspond to the elastic constant and to the mass, respectively. We are interested in computing the two smallest real eigenvalues $\lambda_1$ and $\lambda_2$.

The discretization is applied on a uniform grid with nodes $x_i=i/n, i=0,\dots,n$ and step $h=1/n$, yielding a nonlinear matrix eigenvalue problem of the form $K(\lambda)v=0$ with
$$
K(\lambda) = A-\lambda B+k\frac{\lambda}{\lambda-k/m}C,
$$
where
$$
A=\frac{1}{h}\left(\begin{array}{ccccc}
2 &1&\\
-1&\ddots&\ddots&\\
&\ddots&&2&-1\\
&&&-1&1
\end{array}\right),\quad
B=\frac{h}{6}\left(\begin{array}{ccccc}
4 &1&\\
1&\ddots&\ddots&\\
&\ddots&&4&1\\
&&&1&2
\end{array}\right),\quad
C=e_ne_n^T
$$
and $e_n=[0,\ldots,0,1]^T$. Here we choose $k=2$ and $m=1$. In this case the eigenvalues are known to be $\lambda_1\approx 0.572224720810327$ and $\lambda_2\approx 6.02588212472795$. Interpolation of $\det(K(\lambda))/\det(K(\lambda))'$ on the unit circle followed by FastQZ allows us to approximate $\lambda_1$, see Table \ref{table:mele}. The same technique applied after a suitable translation of $\lambda$ (here $\lambda\rightarrow\lambda-6$) gives approximations for $\lambda_2$.

Note that, since $K(\lambda)$ is a rational function, we should make sure that its pole $\tilde{\lambda}=k/m$ does not lie in the unit disk, otherwise interpolation might not be able to detect the eigenvalues. Experiments with $k=0.01$ and $m=1$, for instance, showed that the first eigenvalue (in this case $\lambda_1\approx 9.90067\cdot 10^{-3}$)  could not be computed via interpolation. 

\end{example}

\begin{table}[h]
\caption{Example \ref{ex:mele}: the table shows the absolute errors on $\lambda_1$ and $\lambda_2$ for several values of $n$ (the number of nodes on the discretization grid). The number of interpolation nodes is taken as $N=100$. The distance between the approximations computed by structured and unstructured QZ is always of the order of the machine epsilon.}\label{table:mele}
\centering
\begin{tabular}{|l|ll|}
\hline\noalign{\smallskip}
$n$&error on $\lambda_1$&error on $\lambda_2$\\
\noalign{\smallskip}\hline\noalign{\smallskip}
$100$&$5.19$e$-3$&$8.48$e$-2$\\
$500$&$1.04$e$-3$&$1.72$e$-2$\\
$1000$&$5.23$e$-4$&$8.61$e$-3$\\
$5000$&$1.05$e$-4$&$1.72$e$-3$\\
\noalign{\smallskip}\hline
\end{tabular}
\end{table}

\section{Conclusions}\label{sec:6}
In this paper we have developed and tested a fast structured version of the double-shift QZ eigenvalue method tailored to a particular class of real matrix pencils. This class includes companion pencils, as well as pencils arising from barycentric Lagrange interpolation. Numerical tests confirm the expected complexity gains with respect to the classical method and show that our fast algorithm behaves as backward stable in practice, while retaining an accuracy comparable to the nonstructured method.

We also propose an application to
nonlinear eigenvalue problems  using interpolation techniques. While preliminary experiments look promising, this approach deserves further investigation, which will be the subject of further work.

\vspace{10pt}

{\bf Acknowledgements:} Thanks to Thomas Mach for useful suggestions concerning the Fortran implementation of Givens transformations.

\section*{Appendix}

In this appendix we give a formal proof of the correctness of the algorithm stated in Section 
\ref{sec:4}.  Specifically,  we prove the following:

\begin{theorem}\label{MAYM12D}
Let $(A,B)\in \mathcal P_N$ be a  matrix pair with an upper Hessenberg matrix
$A=V-\B z\B w^*$ from the class $\mathcal H_N$ and an upper triangular matrix 
$B=U-\B p \B q^*$ from the class ${\mathcal T}_{N}$ with the unitary matrices
$V\in{\mathcal V}_{N},U\in{\mathcal U}_{N}$ and the vectors 
$\B z,\B w,\B p,\B q\in\mathbb R^N$.  Let $p(z)=\alpha + \beta z +\gamma z^2 \in \mathbb R[z]$ be 
a polynomial of degree at most 2.   Let $Q, Z$ be unitary matrices defined as in \eqref{smap22d}, \eqref{smip22d}
where the matrices $Q_i$ and $Z_i$, $1\leq i\leq N-1$, are generated  by the algorithm in Section \ref{sec:4}.
Then  $A_1=Q^** A Z$ and $B_1=Q^* B Z$ are upper Hessenberg and upper triangular, respectively, and, moreover, 
$Q^* p(A B^{-1}) \B e_1=\alpha \B e_1$ for a suitable scalar $\alpha\in \mathbb R$.
\end{theorem}

\begin{proof}
The property $Q^* p(A B^{-1}) \B e_1=\alpha \B e_1$ easily follows by construction.   The proof of the remaining properties 
is constructive by  showing that 
$A_1$ is upper Hessenberg and $B_1$ is upper triangular  and then  providing structured representations of the entries 
of  their  unitary components
$V_1=Q^* V Z$ and $U_1=Q^* U Z$.  We restrict ourselves to consider  $A_1$ and $V_1$ since the 
  computation of  $B_1$ and $U_1$ and of the perturbation vectors  can be treated in a similar way.

We treat $A$ and $V$ as block matrices with entries of sizes 
$m^A_i\times n^A_j,\;i,j=1,\dots,N+3$, where
\be\label{aprmn18}
\begin{gathered}
m^A_1=\dots=m^A_N=1,\;m^A_{N+1}=m^A_{N+2}=m^A_{N+3}=0,\\
n^A_1=0,\;n^A_2=\dots=n^A_{N+1}=1,\;n^A_{N+2}=n^A_{N+3}=0.
\end{gathered}
\end{equation}
Relative to this partition the matrix $V$ has diagonal entries
\be\label{arep18u}
\begin{gathered}
d_V(1)\;\mbox{to be the $1\times0$ empty matrix},\\
d_V(k)=\sg^V_{k-1}=\sg^A_{k-1}+z(k)w(k-1),\;k=2,\dots,N,\\
d_V(N+1),d_V(N+2),d_V(N+3)\;\mbox{to be the $0\times1,0\times0,0\times0$ empty 
matrices},
\end{gathered}
\end{equation}
upper quasiseparable generators
\be\label{aprrl18u}
\begin{gathered}
\hat g_V(k)=g_V(k),\;k=1,\dots,N,\quad
\hat g_V(N+1),g_V(N+2)\;\mbox{to be the $0\times0$ empty matrices},\\
\hat h_V(k)=h_V(k-1),\;k=2,\dots,N+1,\\
\hat h_V(N+2),\hat h_V(N+3)\;\mbox{to be the $0\times0$ empty matrices},\\
\hat b_V(k)=b_V(k-1),\;k=2,\dots,N,\\
\hat b_V(N+1),\hat b_V(N+2)\;\mbox{to be the $r^V_N\times0,0\times0$ empty 
matrices}
\end{gathered}
\end{equation}
and lower quasiseparable generators
\be\label{desc26}
\begin{gathered}
p_V(k)=z(k),\;k=2,\dots,N,\\ p_V(N+1),p_V(N+2)\;\mbox{to be the $0\times 1$
empty matrices},\\
q_V(1)\;\mbox{to be the $1\times0$ empty matrix},\\
q_V(k)=w(k-1),\;k=2,\dots,N+1,\\
a_V(k)=1,\;k=2,\dots,N+1.
\end{gathered}
\end{equation}
Relative to the partition (\ref{aprmn18}) the matrix $A$ is a block upper 
triangular matrix with diagonal entries
\be\label{arep18}
\begin{gathered}
d_A(1)\;\mbox{to be the $1\times0$ empty matrix},\\
d_A(k)=\sg^A_{k-1},\;k=2,\dots,N,\\
d_A(N+1),d_A(N+2)\;\mbox{to be the $0\times1,0\times0$ empty matrices}.
\end{gathered}
\end{equation}
Moreover using (\ref{a}) we obtain upper quasiseparable of the matrix $A$ 
relative to the partition (\ref{aprmn18}) with orders
\be\label{decl1}
r^A_k=r^V_k+1,\;k=1,\dots,N,\quad r^A_{N+1}=r^A_{N+2}=0
\end{equation}
by the formulas
\be\label{aprrl18g}
\begin{gathered}
g_A(k)=\left[\ba{cc}g_V(k)&-z(k)\ea\right],\;k=1,\dots,N,\\
g_A(N+1),g_A(N+2)\;\mbox{to be the $0\times0$ empty matrices},
\end{gathered}
\end{equation}
\be\label{aprrl18bh}
\begin{gathered}
h_A(k)=\left[\ba{c}h_V(k-1)\\w(k-1)\ea\right],\;k=2,\dots,N+1,\\
h_A(N+2),h_A(N+3)\;\mbox{to be the $0\times0$ empty matrix},\\
b_A(k)=\left(\ba{cc}b_V(k-1)&0\\0&1\ea\right),\;k=2,\dots,N,\\
b_A(N+1),b_A(N+2)\;\mbox{to be the $(r^V_N+1)\times0,0\times0$ empty 
matrices}.
\end{gathered}
\end{equation}

Using (\ref{smap22d}) and setting
\be\label{aprela18d}
\begin{gathered}
\tl S^A_1=\tl S^A_2=I_N,\quad \tl S^A_i=\tl Q^*_{i-2},\;i=3,\dots,N+1,\quad
\tl S^A_{N+2}=\tl S^A_{N+3}=I_N,\\
\tl T^A_1=\tl T^A_2=\tl T^A_3=I_N,\quad 
\tl T^A_i=\tl Z_{i-3},\;i=4,\dots,N+2,\quad
\tl T^A_{N+3}=I_N
\end{gathered}
\end{equation}
we get
\be\label{irap18d}
Q^*=\tl S^A_{N+3}\cdots\tl S^A_1,\quad Z=\tl T^A_1\cdots\tl T^A_{N+3}.
\end{equation}
We have
\begin{gather*}
\tl S^A_1={\rm diag}\{S^A_1,I_{N-1}\},\;
\tl S^A_2={\rm diag}\{S^A_2,I_{N-2}\};\\
\tl S^A_k={\rm diag}\{I_{k-2},S^A_k,I_{N-k}\},\; k=2,\dots,N;\\
\tl S^A_{N+1}={\rm diag}\{I_{N-2},S^A_{N+1}\},\;
\tl S^A_{N+2}={\rm diag}\{I_{N-1},S^A_{N+2}\},\;
\tl S^A_{N+3}={\rm diag}\{I_N,S^A_{N+3}\}
\end{gather*}
with
\be\label{aprel21d}
\begin{gathered}
S^A_1=1,\;S^A_2=I_2,\quad S^A_k=Q^*_{k-2},\;k=3,\dots,N+1,\quad 
S^A_{N+2}=1,\\ 
S^A_{N+3}\;\mbox{to be the $0\times0$ empty matrix}
\end{gathered}
\end{equation}
and
\begin{gather*}
\tl T^A_1={\rm diag}\{T^A_1,I_{N}\},\;
\tl T^A_2={\rm diag}\{T^A_2,I_{N-1}\},\;\tl T^A_3={\rm diag}
\{T^A_2,I_{N-2}\};\\
\tl T_k={\rm diag}\{I_{k-4},T^A_k,I_{N-k+1}\},\; k=4,\dots,N+2;\\
\tl T^A_{N+3}={\rm diag}\{I_{N-1},T^A_{N+3}\}
\end{gather*}
with
\be\label{lerpa21d}
\begin{gathered}
T^A_1\;\mbox{to be the $0\times0$ empty matrix},\\
T^A_2=1,\;T^A_3=I_2,\quad
T^A_k=Z_{k-3},\;k=4,\dots,N+2,\quad T^A_{N+3}=1.
\end{gathered}
\end{equation}
We treat the lower Hessenberg matrix $Q^*$ as a block matrix with entries of
sizes $\tau^A_i\times m^A_j,\;i,j=1,\dots,N+3$, where
\be\label{aprr23td}
\tau^A_1=\tau^A_2=0,\quad \tau^A_3=\dots=\tau^A_{N+2}=1,\quad \tau^A_{N+3}=0.
\end{equation}
The matrix $Q^*$ has the representation considered in Lemma 31.1 in \cite{EGH2} with the matrices 
$S_k\;(k=1,\dots,N+2)$ of sizes $(\tau^A_1+r^S_1)
\times m^A_1,\;
(\tau^A_k+r^S_k)\times(m^A_k+r^S_{k-1})\;(k=2,\dots,N+1),\;\tau^A_{N+2}\times
(m^A_{N+2}+r^S_{N+1})$, where 
$$
r^S_1=1,\quad r^S_k=2,\;k=2,\dots,N),\quad r^S_{N+1}=1,\;r^S_{N+2}=0.
$$ 
We treat the upper Hessenberg matrix $Z$ as a block matrix with entries of
sizes $n^A_i\times\nu^A_j,\;i,j=1,\dots,N+2$, where
\be\label{aprr23nd}
\nu^A_1=\nu^A_2=\nu^A_3=0,\quad \nu^A_4=\dots=\nu^A_{N+3}=1.
\end{equation}
The matrix $Z$ has the representation considered in Lemma 31.1 in \cite{EGH2} with the matrices 
$T^A_k\;
(k=1,\dots,N+2)$  of sizes
$n^A_1\times(\nu^A_1+r^*_1),\;(n^A_k+r^*_{k-1})\times(\nu^A_k+r^*_k)\;
(k=2,\dots,N+1),\;(n^A_{N+2}+r^*_{N+1})\times\nu^A_{N+2}$, where 
$$
r^*_1=0,\;r^*_2=1,\quad r^*_k=2,\;k=3,\dots,N+1,\quad r^*_{N+2}.
$$

Now we  apply  the  structured multiplication algorithm for quasiseparable representations 
stated in Corollary 31.2 in \cite{EGH2} in order to determine diagonal entries 
$\tl d_A(k)\;(k=3,\dots,N+2)$ and quasiseparable generators 
$\tl q(j)\;(j=3,\dots,N+1);\;\tl g_A(i)\;(i=3,\dots,N+2)$ of the matrix 
$A_1=Q^*AZ$ as well as auxiliary variables $\bt_k^A,f_k^A,\phi_k^A,
\varphi_k^A$. The matrix $A_1$ is obtained as a block one with entries of sizes
$\tau^A_i\times\nu^A_j,\;i,j=1,\dots,N+3$.

For the variables $\bt_k=\bt^A_k,f_k=f^A_k,\phi_k=\phi^A_k$ used in
Corollary 31.2 we use the partitions
\be\label{ella}
\begin{gathered}
\bt^A_k=\left[\ba{cc}f^A_k&\phi^A_k\ea\right],\;
\phi^A_k=\left[\ba{cc}\varphi^A_k&-\chi_k\ea\right],\;k=2,\dots,N-1,
\end{gathered}
\end{equation}
with the matrices $f^A_k,\varphi^A_k,\chi_k$ of sizes $2\times2,2\times r^V_k,
2\times1$.
For $k=1,\dots,N-2$ combining Corollary 31.2  with
(\ref{aprel21d}), (\ref{lerpa21d}), (\ref{arep18}) and 
(\ref{aprrl18g}),(\ref{aprrl18bh}) we get
\be\label{decc22d}
\begin{gathered}
\left(\ba{cc}\tl d_A(k+3)&\tl g_A(k+3)\\\tl q(k+3)&\beta^A_{k+3}\ea\right)=\\
Q_{k+1}^*
\left(\ba{ccc}f^A_{k+2}&\phi^A_{k+2}h_A(k+2)&\phi^A_{k+2}b_A(k+2)\\
0&\sg^A_{k+2}&g_A(k+3)\ea\right)
\left(\ba{cc}Z_k&0\\0&I_{r^A_{k+3}}\ea\right),\\
k=1,\dots,N-3,\\
\left(\ba{cc}\tl d_A(N+1)&\tl g_A(N+1)\\\tl q(N+1)&\beta^A_{N+1}\ea\right)=
Q_{N-1}^*\left(\ba{cc}f^A_N&\phi^A_Nh_A(N)\ea\right)Z_{N-2}.
\end{gathered}
\end{equation}
Using (\ref{ella}) and (\ref{aprrl18bh}) we get
\be\label{el21}
\phi^A_{k+2}h_A(k+2)=\epsilon^A_{k+2},\quad k=1,\dots,N-2
\end{equation}
and
\be\label{ell21}
\phi^A_{k+2}b_A(k+2)=
\left[\ba{cc}\varphi^A_{k+2}b_V(k+2)&-\chi_{k+2}\ea\right],
\quad k=1,\dots,N-3
\end{equation}
with $\epsilon^A_{k+2}$ as in (\ref{ep}).

Inserting (\ref{el21}), (\ref{ell21}) in (\ref{decc22d}) and using 
(\ref{ella}), (\ref{aprrl18g})  we obtain
\be\label{appr23d}
\begin{gathered}
\left(\ba{cccc}\tl d_A(k+3)&\times&\times&\times\\
\tl q(k+3)&f^A_{k+3}&\varphi^A_{k+3}&-\chi_{k+3}\ea\right)=\\
Q_{k+1}^*\left(\ba{cccc}f^A_{k+2}&\epsilon^A_{k+2}
&\varphi^A_{k+2}b_V(k+2)&-\chi_{k+2}\\
0&\sg^A_{k+2}&g_V(k+3)&-z(k+3)\ea\right)
\left(\ba{cc}Z_k&0\\0&I_{r^A_{k+3}}\ea\right),\\
k=1,\dots,N-3.
\end{gathered}
\end{equation}
From (\ref{appr23d}) using (\ref{aprl18d}) we obtain the relations
\be\label{len}
\left(\ba{c}\tl d_A(k+3)\\\tl q(k+3)\ea\right)=
Q_{k+1}^*\left(\ba{c}\Omega_k(1,1)\\\Omega_k(2,1)\ea\right)
\end{equation}
and
\be\label{appr23dm}
\begin{gathered}
\left(\ba{ccc}\times&\times&\times\\f^A_{k+3}&\varphi^A_{k+3}&-\chi_{k+3}\ea\right)=\\
Q_{k+1}^*\left(\ba{ccc}\Omega_k(1:2,2:3)&\varphi^A_{k+2} &-\chi_{k+2}\\
\Omega_k(3,2:3)&g_V(k+3)&-z(k+3)\ea\right),\quad
k=1,\dots,N-4.
\end{gathered}
\end{equation}
From (\ref{len})  using (\ref{aprle18d}) we have
\be\label{apra23d}
\tl d_A(k+3)=(\sg^A_k)^{(1)},\;k=1,\dots,N-2
\end{equation}
and
\be\label{epra23d}
\tl q(k+3)=0,\;k=1,\dots,N-2.
\end{equation}
The formulas (\ref{aprr23td}) and (\ref{aprr23nd}) mean that 
$(\sg^A_k){(1)},\;k=1,\dots,N-2$ are subdiagonal entries of the matrix $A_1$
(treated as an usual scalar matrix). The equalities
(\ref{epra23d}) imply that $A_1$ is an upper Hessenberg matrix.

Next we apply  the structured multiplication algorithm stated in Lemma 31.1 in \cite{EGH2}
 to compute (block) upper quasiseparable 
generators $\tl g_V(i)\;(i=1,\dots,N+2),\;\tl h_V(j)\;(j=2,\dots,N+3),
\;\tl b_V(k)\;(k=2,\dots,N+2)$ with orders 
$$
\tl r^V_1=r^V_1,\;\tl r^V_2=r^V_2+1,\quad
\tl r^V_k=r^V_k+2,\;k=3,\dots,N,\;\tl r^V_{N+1}=2,\;\tl r^V_{N+1}=1
$$ 
and diagonal entries $\tl d_{V_1}(k)\;(k=1,\dots,N+3)$ of the matrix 
$V_1=Q^*VZ$. The matrix  $V_1$ is obtained as a block one with entries
of sizes $\tau^A_i\times\nu^A_j,\;i,j=1,\dots,N+3$, where the numbers 
$\tau^a_i,\nu^A_j$ are defined in (\ref{aprr23td}), (\ref{aprr23nd}). 

Using Lemma 31.1 and (\ref{aprel21d}), (\ref{lerpa21d}) we obtain that 
\be\label{decb27}
\begin{gathered}
\left(\ba{cc}\tl d_V(k)&\tl g_V(k)\\\times&\G_k\ea\right)=\\
\left(\ba{cc}Q_{k-2}^*&0\\0&1\ea\right)
\left(\ba{ccc}f^V_{k-1}&\phi^V_{k-1}h_V(k-1)&\phi^V_{k-1}b_V(k-1)\\
z(k)\al_{k-1}&\sg^V_{k-1}&g_V(k)\\\al_{k-1}&w^*(k-1)&0\ea\right)
\left(\ba{cc}Z_{k-2}&0\\0&I_{r^V_k}\ea\right),\\
\G_k=\left[\ba{cc}f^V_k&\phi^V_k\\\al_k&0\ea\right],\quad k=4,\dots,N,
\end{gathered}
\end{equation}
together with  the relation (\ref{mmay3hbd}).

From (\ref{decb27}) we find  that the
auxiliary matrices $\al_k\;(k=3,\dots,N)$ satisfy the relations
$$
\al_3=\left(\ba{cc}w(1)&w(2)\ea\right),\quad 
\left(\ba{cc}\times&\al_k\ea\right)=
\left(\ba{cc}\al_{k-1}&w(k-1)\ea\right)Z_{k-3},\;k=4,\dots,N.
$$
Comparing this with (\ref{feura21}), (\ref{feura21d}) we get
\be\label{jana2}
\al_k=\g_{k-1}^*,\quad k=3,\dots,N.
\end{equation}
Thus using  (\ref{jana2}) and (\ref{decb27}) we obtain (\ref{leap18d}),
(\ref{pra18d}).

Next we show that the auxiliary variables $f^A_k,\varphi^A_k\;(k=3,\dots,N)$ 
may be determined via relations (\ref{janj11f}), (\ref{janj1f}). 
 Take $\gamma_2$ as in  (\ref{feura21}) and assume that
for some $k$ with $1\le k\le N-2$ the relations
\be\label{janv4}
f^A_{k+2}=f^V_{k+2}-\chi_{k+2}\gamma_{k+1}^*,\quad 
\varphi^A_{k+2}=\phi_{k+2}^V
\end{equation}
hold. By (\ref{aprl18d}) and (\ref{appr23dm}) we have
\be\label{junirt4}
\begin{gathered}
\left(\ba{ccc}\tl d_A(k+3)&\times&\times\\0&f^A_{k+3}&\varphi^A_{k+3}\ea\right)=\\
Q_{k+1}^*\left(\ba{ccc}f^A_{k+2}&\epsilon^A_{k+2}&\phi^V_{k+2}b_V(k+2)\\
0&\sg^A_{k+2}&g_V(k+3)\ea\right)\left(\ba{cc}Z_k&0\\0&I_{r^V_{k+3}}\ea\right).
\end{gathered}
\end{equation}
Using (\ref{jll8o}) and (\ref{ep})  we get
\be\label{junia4}
\begin{gathered}
\left(\ba{cc}f^A_{k+2}&\epsilon_{k+2}\\0&\sg^A_{k+2}\ea\right)=\\
\left(\ba{ccc}f^V_{k+2}-\chi_{k+2}\gamma_{k+1}^*&
\phi^V_{k+2}h_V(k+2)-\chi_{k+2}w(k+2)\\
z(k+3)\gamma^*_{k+1}-z(k+3)\gamma^*_{k+1} &\sg^V_{k+2}-z(k+3)w(k+2)\ea\right).
\end{gathered}
\end{equation}
Thus combining (\ref{junirt4}) and (\ref{junia4}) together and
using (\ref{leap18d}), (\ref{pra18d}) and (\ref{oct14qqd})
we obtain (\ref{janj1f}).

\end{proof}

\end{document}